\documentclass[a4paper,11pt]{article}

\usepackage{lmodern}
\usepackage[T1]{fontenc}
\usepackage{hyperref}
\usepackage{amsfonts,amsmath,amssymb,amsopn,amsthm}
\usepackage[all]{xy}
\usepackage{vmargin}
\usepackage{makeidx}
\usepackage{graphicx}
\usepackage{tikz}
\usepackage{comment}
\usepackage{mathabx}
\usepackage{mathbbol}
\usepackage[pagewise]{lineno}
\title{Cartier transform and prismatic crystals}
\author{Michel Gros, Bernard Le Stum \& Adolfo Quir\'os
\thanks{Supported by grant PGC2018-095392-B-I00 (MCI/AEI/FEDER, UE).}
}
\date{Version of \today}
\usepackage[ style=alphabetic,citestyle=alphabetic,sorting=nyt,sortcites=true,autopunct=true,babel=hyphen,hyperref=true,abbreviate=false,backref=true]{biblatex}
\AtEveryBibitem{ \clearfield{url} \clearfield{urldate} \clearfield{review} \clearfield{series} \clearfield{doi} \clearfield{isbn} \clearfield{issn} } 
\defbibheading{bibempty}{}
\DeclareNolabel{\nolabel{\regexp{[\p{Z}]+}}}
\newtheorem{thm}{Theorem}[section]
\newtheorem{prop}[thm] {Proposition}
\newtheorem{cor}[thm] {Corollary}
\newtheorem{lem}[thm] {Lemma}
\theoremstyle{definition}
\newtheorem{dfn}[thm] {Definition}

\newenvironment{xmp}[1][Example.]{\begin{trivlist} \item[\hskip \labelsep {\bfseries #1}]}{\end{trivlist}}

\newenvironment{rmk}[1][Remark.]{\begin{trivlist} \item[\hskip \labelsep {\bfseries #1}]}{\end{trivlist}}
\newenvironment{rmks}[1][Remarks.]{\begin{trivlist} \item[\hskip \labelsep {\bfseries #1}]}{\end{trivlist}}

\newcommand{\Addresses}{{% additional braces for segregating \footnotesize
 \bigskip
 \footnotesize

Michel Gros, \textsc{IRMAR, Université de Rennes,
Campus de Beaulieu, 35042 Rennes cedex, France}\par\nopagebreak
\texttt{michel.gros@univ-rennes1.fr}

\medskip

Bernard Le Stum, \textsc{IRMAR, Université de Rennes,
Campus de Beaulieu, 35042 Rennes cedex, France}\par\nopagebreak
\texttt{bernard.le-stum@univ-rennes1.fr}

\medskip

Adolfo Quir\'os, \textsc{Departamento de Mat\'ematicas, Facultad de Ciencias, Universidad Aut\'onoma de Madrid, E-28049 Madrid, Spain}\par\nopagebreak
\texttt{adolfo.quiros@uam.es}

}}

\begin{document}

%%%%%%%%%
\maketitle

%%%%%%%%%%
\begin{abstract}

We show that the abstract equivalence of categories, called Cartier transform, between crystals on the $q$-crystalline and prismatic sites can be locally identified with the explicit local $q$-twisted Simpson correspondence. This establishes four equivalences that are all compatible with the relevant cohomology theories. We restrict ourselves for simplicity to the dimension one situation.

\end{abstract}

%%%%%%%%%%%
\tableofcontents

%%%%%%%%%%%%%%%%%%
%%%%%%%%%%%%%%%%%%
\section*{Introduction}
\addcontentsline{toc}{section}{Introduction}

This article is the continuation of \cite{GrosLeStumQuiros20} and \cite{GrosLeStumQuiros21}. It is devoted to giving the final arguments realizing our project, first outlined in \cite{Gros20}, of putting the local $q$-twisted Simpson correspondence constructed in \cite{GrosLeStumQuiros19} into the perspective of the $q$-crystalline and prismatic sites theories. These two sites, introduced by Bhargav Bhatt and Peter Scholze in \cite{BhattScholze22}, and their beautiful properties are the framework allowing the whole picture to take form. Indeed, we interpret our correspondence, very local in nature, as the explicit description of something much deeper, global and free from arbitrary choices: the Cartier transform, a canonical map between the $q$-crystalline and prismatic topoi (\cite{GrosLeStumQuiros21}, Definition 6.8). This transform had appeared already in the first lines of the proof of theorem 16.18 of \cite{BhattScholze22} that was, in some sense, considering the case of the trivial crystals, and we started to look at it systematically in \cite{GrosLeStumQuiros21}. Our main result takes the form of the commutative diagram of theorem 5.5, that establishes four compatible equivalences between categories of crystals on both the $q$-crystalline and prismatic sites and more down-to-earth categories of modules with a certain type of ``twisted'' connections. Under two of these equivalences, the local $q$-twisted Simpson correspondence corresponds to the functor induced by the general Cartier transform. Moreover, we show that these four equivalences are all compatible with cohomology. Along the way, we explain in particular what kind of twisted de Rham complex computes locally the cohomology of crystals on these two sites. 

Before describing the contents of this work, let us warn the reader that, starting from section 2, most of the results are written only in the dimension one case. We have done this for clarity but, although the higher dimensional case will admittedly require a few more constructions to be written in order to develop the $q$-calculus, no conceptual obstacle prevents the extension in a natural way  of the present results.

Let us come now to the general organization of the article, referring the reader to the main body of the text for the precise setting and statements. In Section \ref{sec1}, we prove that the Cartier transform, a map from the $q$-crystalline topos of a smooth formal scheme $\mathcal{X}$ to the prismatic topos of its base change $\mathcal{X'}$ by the Frobenius of the base is an equivalence of categories (theorem \ref{Liplus}). This result, and the strategy to prove it using ideas of Hidetoshi Oyama (\cite{Oyama17}) and Daxin Xu [(\cite{Xu19}), was foreseen since the writing of \cite{GrosLeStumQuiros20}. Soon after, a complete proof following this strategy was written by Kimihiko Li (\cite{Li21}), but using an additional smart factorization, more precisely via higher level prismatic topoi, of some arrow in our construction we had not noticed. We include here a proof, using the perspective of Frobenius descent for the prismatic topoi, but focusing only on the aspects we need for our own goal (i.e. we only introduce the level 1 prismatic topos) and improving on \cite{Li21} by showing also a cohomological compatibility. In Section \ref{sec2}, we prove (theorem \ref{lift}) a slight enhancement of a result to Bhatt and Scholze about the complete {\emph{faithful}} flatness of some $\delta$-envelope. In revisiting their argument in our setting, we have tried to avoid, as far as possible, the heavy machinery they used to develop their general theory. In Section \ref{sec3}, still in a local situation, we show (proposition \ref{covering}) how to cover the final object of the prismatic topos and deduce (proposition \ref{equi}) from this a description of locally free prismatic crystals as finite projective modules endowed with a {\emph{prismatic stratification}} (Definition \ref{prstr}), a notion that uses crucially the existence of the $\delta$-envelope discussed in the previous section. The reader familiar with classical crystalline cohomology will notice that this mimics exactly the description of crystals as modules endowed with a stratification. Pursuing the analogies with the classical theory, we introduce in Section \ref{sec4}, the {\emph{linearization}} of modules (Definition \ref{lindf}) and of {\emph{prismatic differential operators}} (Definition \ref{deldif}) and prove (proposition \ref{linfunc}) that linearization provides a functor from the category of finite projective modules and prismatic differential operators to the category of prismatic sheaves. In section \ref{sec5}, we come to less formal and much more explicit results that allow us to describe (locally finite free) prismatic crystals as (finite projective) modules endowed with a topologically quasi-nilpotent {\emph{twisted connection of level -1}}. This result (theorem \ref{equiv1}) includes the description of the prismatic envelope of the diagonal as a {\emph{twisted divided polynomial algebra of level -1}}. The same kind of description works also for the $q$-crystalline site and makes striking the analogy with the ``classical'' theory of \cite{Shiho15}. It's then easy to prove (theorem \ref{bigdiag}) the four equivalences of categories we mentioned before. In section \ref{sec6}, we prove the natural compatibilities of our equivalences with the computation of the respective cohomologies, showing in particular that the cohomology of a prismatic crystal is computed by a {\emph{twisted De Rham complex of level -1}}. An analogous result for the $q$-crystalline cohomology is shown to be true using the same technique, and recovers a result of Bhatt and Scholze that they proved in a different way. Finally, an appendix collects some results  about complete flatness (section \ref{Append1}) and derived completeness (section \ref{Append2}) that are not due to us. They are part of the general folklore and can certainly be found here and there but, since we need them at several places, we have thought useful to gather them for the comfort of the reader and have garnished them with some useful examples (sections \ref{Append3}, \ref{Append4}) to keep in mind.

Throughout the article, we fix a prime $p$.

\textbf{Caveat:} In the core of the paper, we will denote by
\[
\widehat M := \varprojlim_{n \in \mathbb N} M/I^{n}
\]
(or $M^{\wedge}$) the \emph{classical} completion of a module $M$ for an $I$-adic topology that should be understood from the context.
In contrast, in the appendix, $\widehat M$ will denote the \emph{derived completion}. 
These choices are made to lighten the notations according to the context.

%%%%%%%%%%%%%%
\section{Prismatic and crystalline sites} \label{sec1}

We briefly recall here the formalism of prismatic and $q$-crystalline sites introduced by Bhatt and Scholze in \cite{BhattScholze22} and enhance some comparison theorems of Li from \cite{Li21} with a discussion of cohomology.

\subsection*{Prismatic crystals}

We refer the reader to \cite{Joyal85} or to section one of \cite{GrosLeStumQuiros21} for the basics on \emph{$\delta$-rings}, and in particular the fact that they carry a frobenius $\phi$.
If $B$ is a $\delta$-ring and $I \subset B$ is an ideal, then $(B,I)$ is called a \emph{$\delta$-pair}.
The $\delta$-pairs form a category in an obvious way.
A $\delta$-pair $(B,I)$ is called a \emph{prism} if $I$ is an invertible ideal, $B$ is derived complete (see definition \ref{defcomp} in the appendix) for the $(p)+I$-adic topology and $p \in I + \phi(I)B$.
Unless otherwise specified, we will always use the $(p)+I$-adic topology.
We will also consider below the corresponding \emph{Nygaard ideal} $\mathfrak N_B := \phi^{-1}(I) \subset B$.
A remarkable feature of prisms is that, if $(B,I) \to (B', I')$ is a morphism, then necessarily $I'=IB'$.
The prism is said to be \emph{bounded} if $B/I$ has bounded $p^{\infty}$-torsion (see definition \ref{inftors} in the appendix).
A prism $(B,I)$ is said to be \emph{orientable} if $I$ is principal and the choice of a generator $d$ is then an \emph{orientation}.
In this situation, the conditions on $I$ are equivalent to $B$ being $d$-torsion free and $d$ being \emph{distinguished} (i.e.\ $\delta(d) \in B^\times$).
We will then call a $B$-module $M$ \emph{bounded} when it is \emph{$d$-torsion free} with bounded $p^{\infty}$-torsion.
Its derived completion is then automatically classically complete (see proposition \ref{compcomp} in the appendix).

The underlying category of the \emph{absolute prismatic site} $\mathbb{\Delta}$ is the category opposite to the category of bounded prisms (but one still writes morphisms in the natural direction).
As is the case with the category $\mathbf{FS}$ of adic formal schemes (with respect to an ideal of finite type), it is not mandatory to assume that the ring $B$ is complete in this definition as long as we allow morphisms between completions.
We make the completeness assumption here in order to comply with the current literature, but it might be convenient later to work with rings that are not complete (in which case, we may simply read $\widehat B$ instead of $B$ when needed).
The category $\mathbb{\Delta}$ does not have a final object or fibered products in general.

Recall (see definition \ref{formflat} in the appendix) that, when $R$ is an adic ring, an $R$-module is said to be \emph{formally (faithfully) flat} if it becomes (faithfully) flat modulo any ideal of definition of $R$.
A morphism of rings $R \to S$ is said to be \emph{formally (faithfully) flat} if it turns $S$ into a formally (faithfully) flat $R$-module.
Pulling back in $\mathbb{\Delta}$ along a formally flat morphism (or, equivalently, completely flat here --- see proposition \ref{flatflat} in the appendix for the orientable case) is always possible.
It follows that formally faithfully flat morphisms define a pretopology that will actually spread to all the categories that we consider below.
This is formalized as follows: if we call a fibered category $T$ over $\mathbb{\Delta}$ \emph{a prismatic site}, then $T$ inherits a pretopology from $\mathbb{\Delta}$ (the coarser topology making the fibration cocontinuous).
Moreover, any morphism of prismatic sites (i.e.\ of fibered categories over $\mathbb \Delta$) $u : T' \to T$ is automatically continuous and cocontinuous, and therefore induces a series of adjoint functors $u_!, u^{-1}, u_*$, the last two of them defining a morphism of topoi.

One may always consider an object of a prismatic site $T$ as a morphism $(B,I) \to T$ if we identify the prism $(B,I)$ with the fibered category that it represents.
A presheaf $E$ on $T$ (we will say a \emph{prismatic presheaf}) is then a family of its \emph{realizations} $E_B$ for all prisms $(B,I)$ over $T$, together with a compatible family of transition maps $E_B \to E_{B'}$ when $(B,I) \to (B',I')$ is a morphism of prisms over $T$.
By definition, it is a sheaf if
\begin{equation} \label{sheaf}
E_B \simeq \ker\left( E_{B'} \rightrightarrows E_{B' \widehat \otimes_B B'}\right)
\end{equation}
whenever $B \to B'$ is cartesian formally faithfully flat.
Note that $E$ is then also a sheaf for the Zariski topology (see remark 2.16 of \cite{BhattScholze22}).
As an example, we can consider the sheaf of rings $\mathcal O_T$ whose realization on $(B,I)$ is the ring $B$.
A morphism of prismatic sites $u : T' \to T$ is automatically a flat morphism of ringed sites: even better, we have $u^{-1}\mathcal O_T = \mathcal O_{T'}$.
If we are given a presheaf of $\mathcal O_T$-modules $E$, then we can linearize the transition maps as $B' \otimes_B E_B \to E_{B'}$.
The presheaf is called a \emph{crystal} if all linerarized transition maps are bijective (be careful that this notion depends on the site and \emph{not} only on the topos).
Note that we do not require that a crystal is a sheaf \emph{a priori}, even if this will be the case in practice.
The inverse image of a crystal by a morphism of prismatic sites is automatically a crystal.

Before going any further, let us recall that formal faithfully flat descent holds for classically complete modules:

%%%%%%%%%%%%%%%%
\begin{lem} \label{faithpro}
Let $R$ be an adic ring with respect to a finitely generated ideal $I$.
Let $B \to C$ be a formally faithfully flat morphism of classically complete $R$-algebras.
Then the classical complete pullback functor $M \mapsto M' := C \widehat \otimes_{B} M$ induces an equivalence between classically complete $B$-modules $M$ and classically complete $C$-modules $M'$ endowed with an isomorphism $C \widehat \otimes_{B} M' \simeq M' \widehat \otimes_{B} C$ satisfying the usual cocycle condition.
Moreover, $M$ is finite projective if and only if $M'$ is.
\end{lem}

\begin{proof}
It follows from \cite[\href{https://stacks.math.columbia.edu/tag/09B8}{Tag 09B8}]{stacks-project} that the category of classically complete $R$-modules is equivalent to the category of inverse systems $\{M_{n}\}_{n \in \mathbb N}$ of $R/I^{n+1}$-modules endowed with a compatible family of isomorphisms $M_{n+1}/I^{n+1}M_{n+1} \simeq M_{n}$.
The first assertion therefore follows from classical faithfully flat descent (\cite[\href{https://stacks.math.columbia.edu/tag/023N}{Tag 023N}]{stacks-project}) modulo $I^{n+1}$.
The second one follows from classical faithfully flat descent for projective modules (\cite[\href{https://stacks.math.columbia.edu/tag/058S}{Tag 058S}]{stacks-project}) and the fact that an $A$-module is finite projective if and only if it is so modulo $I^{n+1}$ for all $n$ (\cite[\href{https://stacks.math.columbia.edu/tag/0D4B}{Tag 0D4B}]{stacks-project}).
\end{proof}

\begin{rmk}
Yichao Tian gives a derived version of lemma \ref{faithpro} in proposition 1.8 of \cite{Tian21}.
\end{rmk}

In the situation of lemma \ref{faithpro}, there exists an exact sequence
\[
0 \to M \to C \widehat \otimes_{B} M \rightrightarrows C \widehat \otimes_{B} C \widehat \otimes_{B} M.
\]

When $M$ is finite projective\footnote{This is equivalent to finitely presented flat.}, we actually have $M' = C \otimes_{B} M$ and the sequence reads
\begin{equation} \label{prodes}
0 \to M \to C \otimes_{B} M \rightrightarrows (C \widehat \otimes_{B} C) \otimes_{B} M.
\end{equation}

%%%%%%%%%%%%%%%%%%
\begin{prop} \label{prisheaf}
For a presheaf $E$ on a prismatic site $T$, the following conditions are equivalent:
\begin{enumerate}
\item
$E$ is a locally finite free sheaf of $\mathcal O_{T}$-modules,
\item
$E$ is a crystal with finite projective realizations.
\end{enumerate}
\end{prop}

We will then call $E$ a \emph{locally finite free prismatic crystal}.

\begin{proof}
Let $(B,I)$ be a prism over $T$.
Assume that $E_{B}$ is finite projective.
It follows from the comments after lemma \ref{faithpro} that, if $C$ is a complete bounded $\delta$-$R$-algebra and $B \to C$ is a formally faithfully flat morphism, then the sequence
\[
0 \to E_{B} \to C \otimes_{B} E_{B} \rightrightarrows (C \widehat \otimes_{B} C) \otimes_{B} E_{B}
\]
is left exact.
Moreover, we can find a complete bounded $\delta$-$R$-algebra $C$ and a formally faithfully flat morphism $B \to C$ such that $C \otimes_{B} E_{B}$ is finite free (use a Zariski covering).
This shows that the second condition implies the first.
Assume conversely that $E$ is a locally finite free prismatic sheaf of $\mathcal O_{T}$-modules.
By definition (of locally), there exists a formally faithfully flat morphism of prisms $B \to C$ such that $E_{C} \simeq C^r$ as \emph{sheaves over $C$}: this implies that any transition map $C_{1} \otimes_{C} E_{C} \to E_{C_{1}}$ for a given $C \to C_{1}$ is bijective.
In particular, $E_{C}$ comes with a descent datum
\[
(C \widehat \otimes_{B} C) \otimes_{C} E_{C} \simeq E_{C \widehat \otimes_{B} C} \simeq E_{C} \otimes_{C} (C \widehat \otimes_{B} C)
\]
as in lemma \ref{faithpro} and therefore defines a finite projective $B$-module $M$.
Since $E$ is a sheaf, we must have $E_{B} = M$ (compare the short exact sequences \eqref{sheaf} and \eqref{prodes}).
It remains to check that the transition maps $B_{1} \otimes_{B} E_{B} \to E_{B_{1}}$ are bijective for all maps $B \to B_{1}$, but this follows by formal faithfully flat descent from the analogous assertion on $C$.
\end{proof}

A prismatic presheaf $E$ will be called \emph{complete} (resp.\ \emph{formally (faithully) flat}) if $E_B$ is classically complete (resp.\ formally (faithfully) flat) for all prisms $(B,I)$.
Note that both conditions together are equivalent to $E_{B}$ being derived complete and completely (faithfully) flat thanks to theorem \ref{compform} in the appendix.
One may also consider the alternative notion of a \emph{complete prismatic crystal} with the requirement $B' \widehat \otimes_{B} E_{B} \simeq E_{B'}$ as in \cite{Li21} or \cite{Tian21} for example.
It then follows from lemma \ref{faithpro} that a complete prismatic crystal is automatically a prismatic sheaf.

\subsection*{Level changing}

We fix now an oriented\footnote{Everything below generalizes to nonoriented bounded prisms, but we are only interested in this case.} bounded prism $(R,d)$, we consider the relative prismatic site $\mathbb{\Delta}(R)$ that it represents and denote by $R_{\mathbb \Delta}$ the corresponding topos.
The site $\mathbb{\Delta}(R)$ is (anti-) equivalent to the category of complete bounded $\delta$-$R$-algebras $B$ and we do not need to mention the ideal (which is automatically equal to $dB$) anymore.
There exist obvious functors

\begin{align*}
&\mathbb{\Delta}(R) \to \mathbf{FS}_{/\mathrm{Spf}(R/d)}, \quad B \mapsto \mathrm{Spf}(B/dB),
\\&\mathbb{\Delta}(R) \to \mathbf{FS}_{/\mathrm{Spf}(R/\mathfrak N_R)}, \quad B \mapsto \mathrm{Spf}(B/\mathfrak N_B),
\end{align*}
where $\mathfrak N_B = \phi^{-1}(dB)$ denotes the Nygaard ideal introduced above.
If $\mathcal X$ is a formal scheme over $R/d$ (resp.\ $R/\mathfrak N_R$), then we can pull back the fibered category that it represents and we obtain a prismatic site $\mathbb{\Delta}(\mathcal X/R)$ (resp.\ $\mathbb{\Delta}^{(1)}(\mathcal X/R)$).
An object is a complete bounded $\delta$-$R$-algebra $B$ endowed with a structural map $\mathrm{Spf}(B/dB) \to \mathcal X$ (resp.\ $\mathrm{Spf}(B/\mathfrak N_{B}) \to \mathcal X$).
A morphism in this site is a morphism of $\delta$-$R$-algebras $B \to B'$ such that the induced map
\[
\mathrm{Spf}(B'/dB') \to \mathrm{Spf}(B/dB) \quad (\mathrm{resp.}\ \mathrm{Spf}(B'/\mathfrak N_{B'}) \to \mathrm{Spf}(B/\mathfrak N_{B}))
\]
 is compatible with the structural maps.
We will denote by $(\mathcal X/R)_{\mathbb{\Delta}}$ (resp.\ $(\mathcal X/R)_{\mathbb{\Delta}^{(1)}}$) the corresponding topos.
The frobenius of $R$ induces a morphism $\overline \phi : R/\mathfrak N_R \hookrightarrow R/d$.
If $\mathcal X$ is a formal scheme over $R/\mathfrak N_R$ and we set $\mathcal X' := \mathcal X \widehat \otimes_{{}_{R/\mathfrak N_R} \nearrow\overline \phi} R/d$, then there exists a \emph{level changing} functor
\[
\rho: \mathbb{\Delta}^{(1)}(\mathcal X/R) \to \mathbb{\Delta}(\mathcal X'/R).
\]
This functor does nothing at the prism level and sends the structural map $\mathrm{Spf}(B/\mathfrak N_{B}) \to \mathcal X$ to
\[
\mathrm{Spf}(B/dB) \to \mathrm{Spf}(B/\mathfrak N_{B}) \widehat \otimes_{{}_{R/\mathfrak N_R} \nearrow\overline \phi} R/d \to \mathcal X \widehat \otimes_{{}_{R/\mathfrak N_R} \nearrow\overline \phi} R/d = \mathcal X',
\]
where the first map is induced by $\overline \phi : B/\mathfrak N_{B} \hookrightarrow B/dB$.
This change of level is a morphism of prismatic sites which is functorial in $\mathcal X/R$ and fully faithful.
But there is more:

%%%%%%%%%%%%%
\begin{thm}[Li]\label{Lithm}
If $(R, d)$ is an oriented prism and $\mathcal X$ is a smooth\footnote{We mean formally smooth and locally finitely presented (or, equivalently here, completely smooth).} formal scheme over $R/\mathfrak N_R$, then the morphism of topoi
\[
\rho : (\mathcal X/R)_{\mathbb{\Delta}^{(1)}} \to ((\mathcal X \widehat \otimes_{{}_{R/\mathfrak N_R} \nearrow\overline \phi} R/d)/R)_{\mathbb{\Delta}}
\]
is an equivalence.
\end{thm}

\begin{proof}
This is theorem 1.17 of \cite{Li21} and we will only sketch it.
According to proposition 4.2.1 of \cite{Oyama17}, it is sufficient to show that the functor is locally surjective at the site level.
If, as above, we denote by $\mathcal X'$ the formal scheme appearing on the right hand side, it means that any object in $\mathbb{\Delta}(\mathcal X'/R)$ can be covered by an object in $\mathbb{\Delta}^{(1)}(\mathcal X/R)$.
In other words, given a $\delta$-$R$-algebra $B'$ and a morphism $\mathrm{Spf}(B'/dB') \to \mathcal X'$, we have to show that there exists a formally faithfully flat morphism of $\delta$-$R$-algebras $B' \to B$ and a morphism $\mathrm{Spf}(B/\mathfrak N_{B}) \to \mathcal X$ such that the diagram
\[
\xymatrix{\mathrm{Spf}(B/dB) \ar[r] \ar[d] & \mathrm{Spf}(B/\mathfrak N_{B}) \widehat \otimes_{{}_{R/\mathfrak N_R} \nearrow\overline \phi} R/d \ar[d] \\ \mathrm{Spf}(B'/dB') \ar[r] & \mathcal X' =\mathcal X \widehat \otimes_{{}_{R/\mathfrak N_R} \nearrow\overline \phi} R/d}
\]
is commutative.
This is a local question and we may therefore assume that there exists a completely regular sequence $\underline f$ in $R[\underline x]$ such that $\mathcal X = \mathrm{Spf}(A)$ with $A = (R/\mathfrak N_R)[\underline x]/(\overline {\underline f})$.
We may then consider, as in \cite{BhattScholze22}, proposition 3.13,  the \emph{prismatic envelope} $S$ (resp.\ $S'$) of
\[
1 \otimes \underline f \in R {}_{{}_{\phi}\nwarrow}\!\!\widehat\otimes_{R} R[\underline x]^\delta \quad \left(\mathrm{resp.}\ \phi(\underline f) \in R[\underline x]^\delta\right)
\]
and set $B := S \widehat \otimes_{S'} B'$.
\end{proof}

As an immediate consequence of theorem \ref{Lithm}, we obtain from proposition \ref{prisheaf} an equivalence between the category of locally finite free $1$-prismatic crystals $E$ on $\mathcal X/R$ (i.e.\ crystals on $\mathbb{\Delta}^{(1)}(\mathcal X/R)$) and the category of locally finite free prismatic crystals $E'$ on $\mathcal X'/R$ (i.e.\ crystals on $\mathbb{\Delta}(\mathcal X/R)$).
There also exists an isomorphism on cohomology
\[
\mathrm R\Gamma((\mathcal X/R)_{\mathbb{\Delta}^{(1)}}, E) \simeq \mathrm R\Gamma((\mathcal X'/R)_{\mathbb{\Delta}}, E')
\]
and therefore
\[
\forall k \in \mathbb Z, \quad \mathrm H^k((\mathcal X/R)_{\mathbb{\Delta}^{(1)}}, E) \simeq \mathrm H^k((\mathcal X'/R)_{\mathbb{\Delta}}, E').
\]

\subsection*{$q$-crystals}

We fix now an indeterminate $q$ and consider the local ring $\mathbb Z[q]_{(p,q-1)}$ endowed with the unique $\delta$-structure such that $\delta(q) = 0$.
We will denote by $(n)_q$ the $q$-analog of an integer $n \in \mathbb Z$.
We define a \emph{$q$-PD-pair} as a $\delta$-pair $(B,J)$, where $B$ is a $(p)_q$-torsion free $\delta$-$\mathbb Z[q]_{(p,q-1)}$-algebra and $J$ an ideal of $B$ such that
\[
\forall f \in J, \quad \phi(f) - (p)_q \delta(f) \in (p)_qJ.
\]
It is called \emph{complete} if $B$ is derived complete (automatically classically complete when $B$ is bounded) and, moreover, $J$ is closed in $B$.
We may then consider the category $q\mathrm{-CRIS}$ whose objects are complete bounded $q$-PD-pairs $(B,J)$ and morphisms again go in the other direction, even if we still write them in the usual way.
There exists a final object $(\mathbb Z_p[[q-1]], (0))$ in this category.
As above, it is not absolutely necessary to assume that the $q$-PD-pairs are complete as long as we allow morphisms between completions, in which case the final object would be $(\mathbb Z[q]_{(p,q-1)}, (0))$.

The forgetful functor
\[
q\mathrm{-CRIS} \to \mathbb{\Delta}, \quad (B,J) \mapsto (B,(p)_q)
\]
is a fibration (in ordered sets) and $q\mathrm{-CRIS}$ is therefore a prismatic site called the \emph{absolute $q$-crystalline site}.
A morphism $(B,J) \to (B',J')$ of $q$-PD-pairs is cartesian when $\overline {JB'} = J'$ (we use the bar to denote the closure).
In particular, this is a covering when the map $B \to B'$ is formally faithfully flat and $\overline {JB'} = J'$.
One defines more generally \emph{a $q$-crystalline site} as a fibered category over $q\mathrm{-CRIS}$ (this is then automatically a prismatic site) and then talk about $q$-presheaves and so on.

We fix a (complete) $q$-PD-pair $(R, \mathfrak r)$ and we consider the $q$-crystalline site $q\mathrm{-CRIS}(R)$  which is represented by $(R, \mathfrak r)$.
Then, the above forgetful functor (the fibration) induces a morphism of prismatic sites $q\mathrm{-CRIS}(R) \to \mathbb{\Delta}(R)$ (where by $\mathbb{\Delta}(R)$ we mean the prismatic site over $(R,(p)_q)$), but the situation is even nicer now.
This forgetful functor has an adjoint $B \mapsto (B,\overline{\mathfrak rB})$ which is a fully faithful morphism of prismatic sites, and also a coadjoint $B \mapsto (B, \mathfrak N_B)$ which is fully faithful and cocontinuous (but is not a morphism of prismatic sites because $\overline {\mathfrak N_BB'} \subsetneq \mathfrak N_{B'}$ in general).
Anyway, we obtain a sequence of adjoint morphisms of topoi
\[
R_{\mathbb{\Delta}} \hookrightarrow R_{q\mathrm{-CRIS}}, \quad R_{q\mathrm{-CRIS}} \to R_{\mathbb{\Delta}}, \quad R_{\mathbb{\Delta}} \hookrightarrow R_{q\mathrm{-CRIS}}.
\]
As above, there also exists a functor
\[
q\mathrm{-CRIS}(R) \to \mathbf{FS}_{/\mathrm{Spf}(R/\mathfrak r)}, \quad (B,J) \mapsto \mathrm{Spf}(B/J).
\]
If $\mathcal X$ is a formal scheme over $R/\mathfrak r$, we can then consider its fiber $q\mathrm{-CRIS}(\mathcal X/R)$ as we did above and we will denote by $(\mathcal X/R)_{q\mathrm{-CRIS}}$ the corresponding topos.
An object of $q\mathrm{-CRIS}(\mathcal X/R)$ is a $q$-PD-pair $(B,J)$ over $(R,\mathfrak r)$ endowed with a structural map $\mathrm{Spf}(B/J) \to \mathcal X$ over $R/\mathfrak r$ and morphisms are defined as usual.
If we write $\mathcal X_1 := \mathcal X \widehat \otimes_{R/\mathfrak r} R/\mathfrak N_R$, then there exists a couple of adjoint functors
\[
\mathbb{\Delta}^{(1)}(\mathcal X_1/R) \overset \alpha \to q\mathrm{-CRIS}(\mathcal X/R) \quad \mathrm{and} \quad q\mathrm{-CRIS}(\mathcal X/R) \overset \beta \to \mathbb{\Delta}^{(1)}(\mathcal X_1/R)
\]
such that $\beta \circ \alpha = \mathrm{Id}$ (and this is natural in $\mathcal X/R$).
The functor $\alpha$ sends the $\delta$-$R$-algebra $B'$ to the $q$-PD-pair $(B', \mathfrak N_{B'})$ and the structural map $\mathrm{Spf}(B'/\mathfrak N_{B'}) \to \mathcal X_1$ to the composite map $\mathrm{Spf}(B'/\mathfrak N_{B'}) \to \mathcal X_1 \hookrightarrow \mathcal X$.
The functor $\beta$ sends the $q$-PD-pair $(B,J)$ to the $\delta$-$R$-algebra $B$ and the structural map $\mathrm{Spf}(B/J) \to \mathcal X$ to the composite
\[
\mathrm{Spf}(B/\mathfrak N_B) \to \mathrm{Spf}(B/J) \widehat \otimes_{R/\mathfrak r} R/\mathfrak N_R \to \mathcal X \widehat \otimes_{R/\mathfrak r} R/\mathfrak N_R = \mathcal X_1.
\]
Both functors are morphisms of prismatic sites (and the first one is fully faithful).
Thus, there exists a pair of adjoint morphisms of topoi
\[
(\mathcal X_1/R)_{\mathbb{\Delta}^{(1)}} \overset \alpha \to (\mathcal X/R)_{q\mathrm{-CRIS}} \quad \mathrm{and} \quad (\mathcal X/R)_{q\mathrm{-CRIS}} \overset \beta \to (\mathcal X_1/R)_{\mathbb{\Delta}^{(1)}} 
\]
such that $\beta \circ \alpha = \mathrm{Id}$ giving rise to a sequence of adjoint functors
\[
\alpha_!, \quad \alpha^{-1} = \beta_!, \quad \alpha_{*} = \beta^{-1}, \quad \beta_*.
\]

The last part of the next lemma is also due to Li (\cite{Li21}, theorem 3.1):

%%%%%%%%%%%
\begin{lem} \label{beta}
If $\mathcal X$ is a formal scheme over $R/\mathfrak r$ then the morphism of topoi
\[
\alpha : (\mathcal X \widehat \otimes_{R/\mathfrak r} R/\mathfrak N_R/R)_{\mathbb{\Delta}^{(1)}} \to (\mathcal X/R)_{q\mathrm{-CRIS}} 
\]
is an exact embedding ($\alpha_*$ is exact and fully faithful).
Moreover, it induces an equivalence between crystals (resp.\ locally finite free crystals) on both sites.
\end{lem}

\begin{proof}
After the previous discussion, only the last assertion needs a proof, and we recall that $\mathcal X_1 := \mathcal X \widehat \otimes_{R/\mathfrak r} R/\mathfrak N_R$.
If $B'$ is the underlying $\delta$-$R$-algebra of an object of $\mathbb{\Delta}^{(1)}(\mathcal X_1/R)$ and $E$ is a sheaf on $\mathcal X/R$, then we have
\[
(\alpha^{-1}E)_{B'} = E_{(B',\mathfrak N_{B'})}.
\]
On the other hand, if $(B,J)$ is the underlying $q$-PD-pair of an object of $q\mathrm{-CRIS}(\mathcal X/R)$ and $E'$ is a sheaf on $\mathcal X_1/R$, then we have
\[
(\alpha_*E')_{(B,J)} = (\beta^{-1}E')_{(B,J)} = E'_{B}.
\]
The assertion relative to crystals is therefore clear.
\end{proof}

If $E'$ corresponds to $E$ in lemma \ref{beta}, then $\mathrm R\alpha_*E' = \alpha_*E' = E$ and therefore also
\[
\mathrm R\beta_*E = \mathrm R\beta_*\mathrm R\alpha_*E' = \mathrm R(\beta \circ \alpha)_*E'=E'.
\]
In particular, we have
\[
\mathrm R\Gamma((\mathcal X/R)_{q\mathrm{-CRIS}}, E) \simeq \mathrm R\Gamma((\mathcal X_1/R)_{\mathbb{\Delta}^{(1)}}, E')
\]
and
\[
\forall k \in \mathbb Z, \quad \mathrm H^k((\mathcal X/R)_{q\mathrm{-CRIS}}, E) \simeq \mathrm H^k((\mathcal X_1/R)_{\mathbb{\Delta}^{(1)}}, E').
\]

%%%%%%%%%%%%%%%%%
\subsection*{Cartier transform}

We now come to the Cartier transform that we introduced in definition 6.8 of \cite{GrosLeStumQuiros21}.
We fix a $q$-PD-pair $(R, \mathfrak r)$, we let $\mathcal X$ be a formal scheme over $R/\mathfrak r$ and we set
\[
\mathcal X' := \mathcal X \widehat \otimes_{{}_{R/\mathfrak r} \nearrow\overline \phi} R/(p)_q.
\]
We consider the functor
\[
C : q\mathrm{-CRIS}(\mathcal X/R) \to \mathbb{\Delta}(\mathcal X'/R)
\]
that sends the $q$-PD-pair $(B,J)$ to the $\delta$-$R$ algebra $B$ and the structural map $\mathrm{Spf}(B/J) \to \mathcal X$ to
\[
\mathrm{Spf}(B/(p)_q) \to \mathrm{Spf}(B/J) \widehat \otimes_{{}_{R/\mathfrak r} \nearrow\overline \phi} R/(p)_q \to \mathcal X \widehat \otimes_{{}_{R/\mathfrak r} \nearrow\overline \phi} R/(p)_q,
\]
where the first map is induced by $\overline \phi : B/J \to B/(p)_qB$.
This is a morphism of prismatic sites that provides us with a morphism of topoi
\[
C : (\mathcal X/R)_{q\mathrm{-CRIS}} \to (\mathcal X'/R)_{\mathbb{\Delta}}.
\]

The next theorem is essentially due to Li (\cite{Li21}).

%%%%%%%%%%
\begin{thm} \label{Liplus}
Let $(R,\mathfrak r)$ be a $q$-PD pair and $\mathcal X$ a smooth formal scheme over $R/\mathfrak r$.
Then $\mathrm RC_*$ and $C^{-1}$ induce an equivalence between locally finite free $q$-crystals $E$ on $\mathcal X$ over $R$ and locally finite free prismatic crystals $E'$ on $\mathcal X' := \mathcal X \widehat \otimes_{{}_{R/\mathfrak r} \nearrow\overline \phi} R/(p)_q$ over $R$.
\end{thm}

\begin{proof}
We have the following decomposition
\[
C : q\mathrm{-CRIS}(\mathcal X/R) \overset \beta \to \mathbb{\Delta}^{(1)}(\mathcal X_{1}/R) \overset \rho \to \mathbb{\Delta}(\mathcal X'/R),
\]
where $\mathcal X_{1} := \mathcal X \widehat \otimes_{R/\mathfrak r} R/\mathfrak N_R$, $\beta$ is as in lemma \ref{beta} and $\rho$ is as in theorem \ref{Lithm} (but with $\mathcal X_{1}$ instead of $\mathcal X$).
We know from proposition \ref{prisheaf} that a finite locally free crystal is the same thing as a finite locally free sheaf and therefore a topos-theoretic notion.
Thanks to theorem \ref{Lithm}, we are therefore reduced to proving our assertion with $C$ replaced by $\beta$ and this then follows from lemma \ref{beta}.
\end{proof}

As a consequence, we will have in the setting of theorem \ref{Liplus},
\[
\mathrm R\Gamma((\mathcal X/R)_{q\mathrm{-CRIS}}, E) \simeq \mathrm R\Gamma((\mathcal X'/R)_{\mathbb{\Delta}}, E')
\]
(for the trivial crystal, this is due to Bhatt and Scholze (\cite{BhattScholze22}, theorem 16.18)) and therefore also
\[
\forall k \in \mathbb Z, \quad \mathrm H^k((\mathcal X/R)_{q\mathrm{-CRIS}}, E) \simeq \mathrm H^k((\mathcal X'/R)_{\mathbb{\Delta}}, E').
\]

%%%%%%%%%%%%%%%%%
\section{Derived and completed envelopes} \label{sec2}

We review some material from \cite{BhattScholze22} that will be needed afterwards.
We shall try to stay in the realm of ``classical'' constructions and avoid arguments from non abelian derived algebra (see however the thesis \cite{Mao21} of Zouhang Mao for a detailed approach using animated rings).

We let $R$ be a commutative ring and fix some $d \in R$.
An $R$-module $M$ will always be endowed with its $(p,d)$-adic topology and we will write $\overline M = M/dM$.

At some point, we will assume that $R$ is a $\delta$-$\mathbb Z_{(p)}$-algebra and that $d$ is distinguished in $R$.

We refer the reader to the appendix for basic results concerning complete flatness and derived completions.
Let us just recall from definition \ref{compfl} in the appendix that a complex $M^\bullet$ of $R$-modules is said to be \emph{completely (faithfully) flat} if the complex $\overline R/p\overline R \otimes^{\mathrm L}_R M^\bullet$ of $\overline R/p\overline R$-modules is discrete (faithfully) flat.

%%%%%%%%%%%%%%%%%
\begin{dfn} \label{comfr}
An element $g$ in an $R$-algebra $B$ is \emph{completely (faithfully) regular over $R$} if the complex $[B \overset g \to B]$ is completely (faithfully) flat over $R$.
\end{dfn}

\begin{rmks}
\begin{enumerate}
\item The complex $[B \overset g \to B]$ is the Koszul complex $\mathrm{Kos}(B, g)$ (see \cite[\href{https://stacks.math.columbia.edu/tag/0623}{Tag 0623}]{stacks-project}).
We have
\[
[B \overset g \to B] \simeq \mathbb Z \otimes^{\mathrm L}_{\mathbb Z[x]} B,
\]
with $x \mapsto 0$ on the left and $x \mapsto g$ on the right.
In particular,  definition \ref {comfr} generalizes in a straightforward way to the case of a complex of $R$-modules $M^\bullet$ endowed with a sequence of commuting endomorphisms $g_{1}, \ldots, g_{d}$ (which may then be seen as a complex of $\mathbb Z[x_{1}, \ldots, x_{d}]$-modules).
\item It follows from the last remarks before definition \ref{defcomp} in the appendix that the property of being completely (faithfully) regular is stable under derived pullback (and detected by derived pullback under a completely faithfully flat morphism).
It also follows from proposition \ref{flatder} in the appendix that it is invariant under derived completion.
\item
Assume $B$ is completely flat over $R$.
Then, $g$ is completely (faithfully) regular over $R$ if and only if $g$ becomes regular in $\overline B/p\overline B$ and $(\overline {B/gB})/p(\overline {B/gB})$ is (faithfully) flat over $\overline R/p\overline R$.
In particular, the property then only depends on the class of $g$ modulo $(p,d)$.
\item If $B$ is completely flat over $R$ and $g \in B$ is completely (faithfully) regular over $R$, then any power $g^k$ of $g$ is also completely (faithfully) regular over $R$ (use the previous remark and induction).
As a consequence (and using the previous remark again), $\phi(g)$ will also be completely (faithfully) regular over $R$ when $R$ is a $\delta$-ring.
\item We will need this notion of complete regularity only in the very simple following case: if $B$ is a an $R$-algebra, then the variable $\xi$ is completely faithfully regular over $B$ in the polynomial ring $B[\xi]$.
\end{enumerate}
\end{rmks}

So far, we haven't used the fact that $R$ is a $\delta$-ring (nor that $d$ is distinguished).
We will keep our notations from \cite{GrosLeStumQuiros20} and denote\footnote{This is the same thing as $R\{y\}$ in \cite{BhattScholze22}, but we'd rather keep curly brackets for convergent series.} by $R[x]^\delta$ the $\delta$-envelope of $R[x]$, that is, the polynomial ring
\[
R[\{x_{k}\}_{k \in \mathbb N}]
\]
endowed with $\delta(x_{k}) = x_{k+1}$ for $k \in \mathbb N$.
We will systematically use the fact (\cite{BhattScholze22}, lemma 2.11) that the frobenius $\phi :R[x]^\delta \to R[x]^\delta$ is faithfully flat.

If $B$ is a $\delta$-$R$-algebra and $g \in B$, we will write
\[
B[g/d]^{\delta} := R[w]^\delta \otimes^\mathrm{L}_{R[x]^\delta} B
\]
with $x \mapsto dw$ on the left and $x \mapsto g$ on the right (note that this is the derived version).
Observe also that, if $B \to C$ is a morphism of $\delta$-$R$-algebras, then
\[
C \otimes^\mathrm{L}_{B} B[g/d]^{\delta} \simeq C[g/d]^{\delta}
\]
if we still denote by $g$ the image of $g$ in $C$.

We will need the following particular case of proposition 3.13 of \cite{BhattScholze22} (actually we need the faithful version which is not stated in loc.\ cit.):

%%%%%%%%%%%%%%%
\begin{thm}[Bhatt-Scholze] \label{lift}
Let $R$ be a $\delta$-$\mathbb Z_{(p)}$-algebra and $d \in R$ distinguished.
If $B$ is a $\delta$-$R$-algebra and $g \in B$ is completely (faithfully) regular over $R$, then the complex $B[g/d]^{\delta}$ is completely (faithfully) flat over $R$.
\end{thm}

\begin{proof}
Since the frobenius $\phi : \mathbb Z_{(p)}[x]^\delta \to \mathbb Z_{(p)}[x]^\delta$ is faithfully flat, 
we can first replace $R$ with
\[
\mathbb Z_{(p)}[x]^\delta {}_{{}_{\phi}\nwarrow}\!\!\otimes_{\mathbb Z_{(p)}[x]^\delta} R,
\]
where, on the right, $x \mapsto d$, and assume that there exists $e \in R$ such that $d=\phi(e)$.
For the same reason, we can replace $B$ with
\[
R[x]^\delta {}_{{}_{\phi}\nwarrow}\!\!\otimes_{R[x]^\delta} B
\]
where, on the right, $x \mapsto g$ and assume that there exists $h \in B$ such that $g = \phi(h)$.

If we denote by $\mathbb Z_{(p)}[x]^\mathrm{PD}$ the divided power envelope of $(x)$ in $\mathbb Z_{(p)}[x]$, we can consider the obvious factorization of the projection:
\[
\mathbb Z_{(p)}[x] \to \mathbb Z_{(p)}[x]^\mathrm{PD} \to {\mathbb Z_{(p)}[x]^\mathrm{PD}}/p {\mathbb Z_{(p)}[x]^\mathrm{PD} }\to \mathbb F_{p}.
\]
We can pull it back (classically) to $R$ (and to $\overline R$ on the right) along the map $x \mapsto e$ and, writing
\[
S := R \otimes_{\mathbb Z_{(p)}[x] } \mathbb Z_{(p)}[x]^\mathrm{PD},
\]
obtain a factorisation of the projection:
\[
R \to S \to \overline S/p\overline S \to \overline R/p\overline R.
\]
It is then sufficient to prove that the complex $S \otimes^{\mathrm L}_R B[g/d]^{\delta}$ is completely (faithfully) flat over $S$.
We can deduce from the interpretation of divided powers in terms of $\delta$-structures (lemma 2.36 of \cite{BhattScholze22}) that

\begin{align*}
S &= R \otimes_{\mathbb Z_{(p)}[x] } \mathbb Z_{(p)}[x]^\mathrm{PD}
\\& \simeq R \otimes_{\mathbb Z_{(p)}[x]^\delta } \mathbb Z_{(p)}[x]^\mathrm{\delta,PD} \\ &\simeq R \otimes_{\mathbb Z_{(p)}[x]^\delta } \mathbb Z_{(p)}[x, \phi(x)/p]^\delta
\\& \simeq R[d/p]_0^\delta.
\end{align*}
If we set $u := d/p \in S$, then $\delta(d) =\delta(pu) = \delta(p)u^p + p\delta(u)$.
Since $\delta(d) , \delta(p) \in R^\times$, we can write $u^p = v(1 - f)$ with $v \in R^\times$ and $f \equiv 0 \mod p$.
It follows that $u$ acts invertibly modulo $p$ and therefore
\[
\overline S/p\overline S \otimes^{\mathrm L}_R B[g/d]^{\delta} \simeq \overline S/p\overline S \otimes^{\mathrm L}_R B[g/p]^{\delta}.
\]
It is thus sufficient to show that $S \otimes^{\mathrm L}_R B[g/p]^{\delta}$ is completely (faithfully) flat over $S$ or even that $B[g/p]^{\delta}$ is completely (faithfully) flat over $R$.
Since we were careful enough to make sure that we can write $g = \phi(h)$, this follows from lemma \ref{pcase} below (with $h$ in the place of $g$).
\end{proof}

%%%%%%%%%%%%%%%
\begin{lem} \label{pcase}
If $B$ is a $\delta$-$R$-algebra and $g \in B$ is such that $\phi(g)$ is completely (faithfully) regular over $R$, then the complex $B[\phi(g)/p]^{\delta}$ is completely (faithfully) flat over $R$.
\end{lem}

\begin{proof}
Let us choose flat simplicial resolutions $R_{\bullet}$ over $\mathbb Z_{(p)}$ and $B_{\bullet}$ over $R_{\bullet}[x]^\delta$ of $R$ and $B$ respectively in the category of $\delta$-rings.
Fix some $i \in \mathbb N$.
The image $g_i$ of $x$ in $B_i$ is regular modulo $p$ and $B_i$ is $p$-torsion free.
Therefore, it follows from corollary 2.39 of \cite{BhattScholze22} that there exists an  isomorphism with the divided power envelope of $g_i$ in $B_i$:
\[
B_i[\phi(g_i)/p]^{\delta} \simeq B_i^\mathrm{PD}.
\]
Now, we consider the free $R$-module $F := \oplus_{n \in \mathbb N} Rv_{n}$ and send $v_{n}$ to $g_i^{[pn]}$.
It induces an isomorphism:
\[
 B_i/\phi(g_i) B_i \otimes_{R} F/pF \simeq B_i^\mathrm{PD}/pB_i^\mathrm{PD}.
\]
If we identify a simplicial complex with the corresponding chain complex via the Dold-Kan correpondence, then we have
\[
B_\bullet/\phi(g_\bullet)B_\bullet \simeq [B \overset {\phi(g)} \to B] \quad \mathrm{and} \quad B_{\bullet}[\phi(g_{\bullet})/p]^{\delta} \simeq B[\phi(g)/p]^{\delta},
\]
and therefore
\[
\overline R/p\overline R \otimes_{R}^{\mathrm L} [B \overset {\phi(g)} \to B] \otimes_{R} F \simeq \overline R/p\overline R \otimes_{R}^{\mathrm L} B[\phi(g)/p]^{\delta}.
\]
Our hypothesis implies that the left hand side is (faithfully) flat over $\overline R/p\overline R$ and the same thing therefore also holds for the right hand side.
\end{proof}

%%%%%%%%%%%%
\begin{rmks}
\begin{enumerate}
\item The proof of theorem \ref{lift} (and lemma \ref{pcase}) follows exactly the same pattern as in \cite{BhattScholze22} (lemma 2.43, corollary 2.44 and proposition 3.13).
It is however self-contained and only requires the interpretation of classical divided powers in term of $\delta$-structures
\item Lemma \ref{pcase} is still valid if we replace the condition that $\phi(g)$ is completely (faithfully) regular by the same condition on $g$ as long as $B$ is completely flat, because the condition is then stronger.
\item Of course, theorem \ref{lift}, as well as the lemma, hold more generally for a completely (faithfully) regular sequence, but we shall only need this simple case.
They also hold in the nonoriented situation: we can replace $(d)$ with an invertible ideal.
\item For a more general statement than theorem \ref{lift}, the reader may consider proposition 5.49 of \cite{Mao21}.
\end{enumerate}
\end{rmks}

If $B$ is a $\delta$-$R$-algebra and $g \in B$, we will set $B[g/d]_0 = R[w] \otimes_{R[x]} B$, with $x \mapsto dw$ on the left and $x \mapsto g$ on the right (we use the subindex $0$ to make clear that we do not mean a derived version).
Then, if $B[g/d]_0^{\delta}$ denotes the $\delta$-envelope of $B[g/d]_0$ (in contrast with the derived version), we have
\[
B[g/d]_0^{\delta} \simeq \mathrm H^0\left( B[g/d]^{\delta}\right) = R[w]^\delta \otimes_{R[x]^\delta} B.
\]
We will denote by $B[g/d]_0^{\delta,\wedge}$ its classical completion.

%%%%%%%%%%%%%%%%%%%%%%%%%%%%%%%%%%%
\begin{cor} \label{extracor}
In the situation of theorem \ref{lift}, if $R$ is bounded (with respect to $p$ modulo $d$: see definition \ref{defbound} in the appendix), then $B[g/d]_0^{\delta, \wedge}$ is a formally (faithfully) flat $\delta$-$R$-algebra.
In particular, it is bounded.
\end{cor}

\begin{proof}
If (only in this proof) we denote derived completion by $\mathrm L\wedge$ (in contrast with classical completion $\wedge$), then it follows from lemma \ref{tenschek} in the appendix that
\[
B[g/d]^{\delta, \mathrm L\wedge} \simeq B[g/d]^{\delta, \wedge} \simeq  R[w]^\delta \widehat \otimes_{R[x]^\delta} B \simeq B[g/d]_0^{\delta, \wedge}.
\]
Our assertion therefore follows from theorem \ref{lift} and theorem \ref{compform} in the appendix.
\end{proof}

\begin{rmk}
The equality $B[g/d]^{\delta, \wedge} = B[g/d]_0^{\delta, \wedge}$ also holds when $R$ is not bounded.
However, even when $R$ is bounded, we cannot say much about $B[g/d]_0^{\delta}$ itself, or its derived completion either.
Classical completion is necessary.
\end{rmk}

%%%%%%%%%%%%%%%%%%%%%%%%
\section{Prismatic stratifications} \label{sec3}

We will introduce in the end the notion of a prismatic stratification on a module and show that, in a local situation, it corresponds to that of a prismatic crystal (see also section 3.1 of \cite{MorrowTsuji20}).

We keep the assumptions and notations from the previous section.
More precisely, we let $R$ be a $\delta$-$\mathbb Z_{(p)}$-algebra $R$ with a distinguished element $d$ so that $(R,d)$ will be our base prism.
An $R$-module $M$ will always be endowed with its $(p,d)$-adic topology and we will write $\overline M = M/dM$. 

At some point, we will assume that $R$ is bounded, we will consider a complete $R$-algebra $A$ with a topologically \'etale coordinate $x$ and we will let $\mathcal X := \mathrm{Spf}(\overline A)$.

%%%%%%%%%
\begin{dfn}
The \emph{(bounded) prismatic envelope} of an ideal $J$ in a bounded $\delta$-$R$-algebra $B$ is a complete bounded $\delta$-$R$-algebra $C$ endowed with a morphism $B \to C$ which is universal for the property $JC \subset dC$.
\end{dfn}

We may also say that the prism $(C, dC)$ is the prismatic envelope of $(B, J)$.
It is unclear for the authors under which generality such an envelope exists, but we will show its existence and explicit description in our case of interest (see proposition \ref{pri} below).

\begin{xmp}
If $R$ is bounded in theorem \ref{lift}, then $B[g/d]^{\delta,\wedge}$ is the prismatic envelope of $(g)$ in $\widehat B$.
More precisely, this is a complete bounded $\delta$-$R$-algebra which is isomorphic to $B[g/d]_0^{\delta,\wedge}$ thanks to corollary \ref{extracor}.
The assertion therefore follows from the universal property of $B[g/d]_0^{\delta,\wedge}$.
\end{xmp}

Let us recall the following notation from section 1 of \cite{GrosLeStumQuiros20}: if $A$ is a $\delta$-ring and $I \subset A$ any ideal, then we denote by $I_{\delta}$ the $\delta$-ideal generated by $I$.

%%%%%%%%%%
\begin{dfn} \label{prispol}
If $B$ is a $\delta$-$R$-algebra and $x \in B$, then the \emph{ring of prismatic polynomials on $B$} (with respect to $x$) is
\[
B[\omega]^\delta_{d} := B[\omega]^\delta/(\delta(x+d\omega))_{\delta}
 \]
 (where $\omega$ is an indeterminate).
\end{dfn}

\begin{rmk}
The ring of prismatic polynomials is universal among $\delta$-$B$-algebras $C$ endowed with some $h \in C$ such that $\delta(x+dh) = 0$ (if we still denote by $x$ the image of $x$ in $C$).
\end{rmk}

\begin{prop} \label{prismenv}
If $B$ be a bounded $\delta$-$R$-algebra and $x \in B$, then $B[\omega]^{\delta,\wedge}_{d}$ is a complete formally faithfully flat $\delta$-$B$-algebra.
Moreover, if we endow the polynomial ring $B[\xi]$ with the unique $\delta$-structure such that $\delta(x + \xi) =0$, then $ B[\omega]^{\delta,\wedge}_{d}$ is the prismatic envelope of $\xi$ in $B[\xi]$.
\end{prop}

\begin{proof}
Recall that we have defined for $\xi \in B[\xi]$ the rings $B[\xi][\xi/d]_0$ and $B[\xi][\xi/d]_0^{\delta}$ where the subindex $0$ means that we do not consider the derived version.
Then, since $B$ is $d$-torsion free, there exists an isomorphism
\[
B[\omega]^{\delta}_{d} \simeq B[\xi][\xi/d]_0^{\delta}, \quad \omega \mapsto \xi/d.
\]
It then follows from corollary \ref{extracor} that $B[\omega]^{\delta,\wedge}_{d}$ is completely faithfully flat over $B$.
The last assertion follows from the universal property of $B[\xi][\xi/d]_0$. 
\end{proof}

We will need the following elementary result:

%%%%%%%%%%%%%%%%
\begin{lem} \label{dletext}
Let $A' \to A$ be a formally étale morphism of $\delta$-rings and let $B$ classically $p$-adically complete ring.
For a ring morphism $A \to B$ to be a morphism of $\delta$-rings, it is sufficient that the composite map $A' \to B$ is a morphism of $\delta$-rings.
\end{lem}

\begin{proof}
Recall that if we denote by $W_{1}(S)$ the ring of Witt vectors of length $2$ on a ring $S$, then it is equivalent to give a $\delta$-structure on $S$ and to give a section of the projection $W_{1}(S) \to S$.
In our situation, we have
\[
\xymatrix{
W_1(A') \ar[r] \ar[d] & W_1(A) \ar[r] \ar[d] & W_1(B) \ar[d] \\
A' \ar[r] \ar@/^1pc/[u] & A \ar[r]  \ar@/^1pc/[u]& B.  \ar@/^1pc/[u]
}
\]
We want to show that the right hand side square involving the sections is commutative and we know that both other squares are.
Since $A' \to A$ is formally étale and both maps $A \to W_1(B)$ coincide, not only on $A'$, but also modulo $p$, which is topologically nilpotent, they must be equal on $A$.
\end{proof}

We assume now that $R$ is bounded and consider a complete $R$-algebra $A$ with a topologically \'etale\footnote{Or equivalently completely étale, since $R$ is bounded.} coordinate $x$.
We mean that there exists a topologically finitely presented formally \'etale morphism $R[x] \to A$.
We let $\mathcal X := \mathrm{Spf}(\overline A)$.
We endow $A$ with the unique structure of $\delta$-$R$-algebra such that $\delta(x)=0$.
This is a bounded prism on $\mathcal X$ over $R$.

%%%%%%%%%%%%%%
\begin{prop} \label{pri}
If $B$ is a complete bounded $\delta$-$R$-algebra and $A \to B$ is a morphism of $R$-algebras, then $B[\omega]^{\delta,\wedge}_{d}$ is the prismatic envelope of the kernel $J$ of multiplication $B \otimes_{R} A \twoheadrightarrow B$.
\end{prop}

It is \emph{not} assumed here that $A \to B$ is a morphism of $\delta$-rings.
Also the ``$x$'' hidden in the definition of the ring of prismatic polynomials is the image of $x$ in $B$.

\begin{proof}
Since $x$ is a topologically \'etale coordinate, it follows from lemma \ref{dletext} that the $\delta$-morphism
\[
R[x] \to B[\omega]^{\delta}_{d}, \quad x \mapsto x + d\omega
\]
extends to a $\delta$-morphism $\theta: A \to B[\omega]^{\delta,\wedge}_{d}$.
This provides a $\delta$-morphism $B \otimes_{R} A \to B[\omega]^{\delta,\wedge}_{d}$.
Since $\overline \theta$ factors through $\overline B$, the induced map
\[
\overline B \otimes_{\overline R} \overline A \to \overline {B[\omega]^{\delta,\wedge}_{d}}
\]
also factors through $\overline B$.
It means that the image of $J$ is contained in $d B[\omega]^{\delta,\wedge}_{d}$.
Assume now that we are given a complete bounded $\delta$-$R$-algebra $C$ and a morphism $B \otimes_{R} A \to C$ such that $JC \subset dC$.
We endow $B[\xi]$ with the unique $\delta$-structure such that $\delta(x + \xi) = 0$ and consider the $\delta$-morphism
\[
B[\xi] \to B \otimes_{R} A, \quad \xi \to 1 \otimes x - x \otimes 1.
\]
Proposition \ref{prismenv} implies that the composite map $B[\xi] \to B \otimes_{R} A \to C$ extends uniquely to a $\delta$-morphism $B[\omega]^{\delta,\wedge}_{d} \to C$ (with $\xi = d\omega$).
Since $\xi$ is a topologically \'etale coordinate for $B \otimes_R A$ over $B$, the usual argument shows that this morphism does extend the original one.
\end{proof}

\begin{rmks}
\begin{enumerate}
\item It is important to notice that the $\delta$-morphism $\theta : A \to  B[\omega]^{\delta,\wedge}_{d}$ that we consider here is the \emph{Taylor map} $x \mapsto x + d\omega$.
We will systematically write $ B[\omega]^{\delta,\wedge}_{d} \otimes'_{A} - $ when we want to insist on the fact that we use this Taylor map as structural map.
\item
In the particular case $B=A$, we see that $A[\omega]^{\delta,\wedge}_{d}$ is the prismatic envelope of the diagonal in the product $P := A \otimes_{R} A$.
There exists a (left) structure $A \to  A[\omega]^{\delta,\wedge}_{d}$ which is the naive map $x \mapsto x$ and a right structure $\theta : A \to  A[\omega]^{\delta,\wedge}_{d}$ which is the \emph{Taylor map} $x \mapsto x + d\omega$.
\item If, besides the map $A \to B$, we have any morphism $B \to C$ of complete bounded $\delta$-$R$-algebras, then
\[
C [\omega]^{\delta}_{d} \simeq C \otimes_{B} B [\omega]^{\delta}_{d}.
\]
It follows that
\[
B [\omega]^{\delta,\wedge}_{d} [\omega']^{\delta,\wedge}_{d} \simeq B [\omega]^{\delta,\wedge}_{d} \widehat \otimes_{B} B [\omega]^{\delta,\wedge}_{d} \simeq B [\omega]^{\delta,\wedge}_{d} \widehat \otimes'_{A}  A[\omega]^{\delta,\wedge}_{d}
\]
(using both structural maps).
\end{enumerate}
\end{rmks}

%%%%%%%%%%%%%%%%%%%
\begin{cor} \label{product}
If $B$ is a bounded prism on $\mathcal X$ over $R$, then $ B[\omega]^{\delta,\wedge}_{d}$ is the product of $B$ by $A$ in the prismatic site of $\mathcal X$ over $R$.
\end{cor}

\begin{proof}
We choose a lifting in $B$ of the image in $\overline B$ of the class $\overline x$ of $x$ in $\overline A$ and we apply proposition \ref{pri}.
\end{proof}

%%%%%%%%%%%%%%%%%%%
\begin{prop} \label{covering}
The prism $A$ is a covering of (the final object of the topos associated to) $\mathbb{\Delta}(\mathcal X/R)$ for the \emph{flat} topology.
\end{prop}

\begin{proof}
It means that, in the prismatic site of $\mathcal X$ over $R$, if we are given any $B$, then there exists a formally faithtully flat morphism of prisms $B \to B'$ and a morphism $A \to B'$.
Since we know from corollary \ref{extracor} that the product $B [\omega]^{\delta,\wedge}_{d}$ of $B$ by $A$ is formally faithfully flat over $B$, we can simply choose $B' := B [\omega]^{\delta,\wedge}_{d}$.
\end{proof}

%%%%%%%%%%%%%%%
\begin{rmks}
\begin{enumerate}
\item The result will hold more generally with more coordinates and gives a general way to describe locally the prismatic site of any smooth formal scheme $\mathcal X$.
\item The result also holds locally when $(R,I)$ is not orientable: one may always find an orientation after a formally faithfully flat extension of the base.
\item The same proof with $d$ replaced by $p$ shows that $(\mathbb Z_{p}[[x]], x-p)$, endowed with the unique $\delta$-structure such that $\delta(x) = 0$, covers the absolute prismatic site.
\item There exists more sophisticated versions of proposition \ref{covering}, such as proposition 5.56 of \cite{Mao21} or lemma 3.2 of \cite{Tian21}, for example.
The strategy is always the same.
\end{enumerate}
\end{rmks}

In the next definition, we should perhaps say prismatic hyper-stratification, but we drop the prefix because there is nothing like an non-hyper-stratification in this theory.

%%%%%%%%%%%%%%%%%%
\begin{dfn} \label{prstr}
A \emph{prismatic stratification} on an $A$-module $M$ is an isomorphism
\[
 A[\omega]^{\delta,\wedge}_{d} \otimes'_{A} M \simeq M \otimes_{A}  A[\omega]^{\delta,\wedge}_{d}
\]
(in which $\otimes'$ indicates again that we use $\theta$ as structural map) satisfying the usual cocycle condition.
\end{dfn}

%%%%%%%%%%%%%%
\begin{prop} \label{equi}
The category of locally finite free prismatic crystals on $\mathcal X$ is equivalent to the category of finite projective $A$-modules endowed with a prismatic stratification.
\end{prop}

\begin{proof}
Follows from proposition \ref{covering}.
\end{proof}

\begin{rmks}
\begin{enumerate}
\item
In the same way, one shows that the category of \emph{absolute} prismatic crystals is equivalent to the category of $\mathbb Z_{p}[[x]]$-modules endowed with a ``prismatic stratification'', using this time the ring
\[
\mathbb Z[x, \omega]_{p}^{\delta} := \mathbb Z[x, \omega]^{\delta}/(\delta(x+p\omega))_{\delta}.
\]
\item Proposition 3.7 in \cite{Tian21} provides a generalization of proposition \ref{equi}.
\item
All these results are rather theoretical: it is in general very hard to do any computation with the description we have given of a prismatic envelope.
We will see how to remedy this in a more specific situation in section \ref{sec5}.
\end{enumerate}
\end{rmks}

%%%%%%%%%%%
\section{Prismatic differential operators} \label{sec4}

We will explain a standard technic (see for example construction 6.9 in \cite{BerthelotOgus78}) that can be used in order to compute the cohomology of a prismatic crystal.
For this purpose, we need to introduce the notion of a prismatic differential operator.
We keep the notations as in the previous sections.
Thus, we assume that $R$ is a bounded $\delta$-$\mathbb Z_{(p)}$-algebra and that $d \in R$ is distinguished.
Any $R$-module $M$ is endowed with its $(p,d)$-adic topology and we write $\overline M$ for $M/dM$.

At some point, we will consider a complete bounded $\delta$-$R$-algebra $A$ and write $\mathcal X := \mathrm{Spf}(\overline A)$.
At some point, we will also assume that $A$ has a topologically \'etale coordinate $x$.

We begin with some very general considerations.
Let $T$ be any site (for example a prismatic site) and $\widetilde T$ denote the corresponding topos.
If we denote by $\mathbb 1 := \{0\}$ the final category, then there exists (by definition) a unique functor $e_{T} : T \to \mathbb 1$.
It is obviously cocontinuous and, if we identify the category of (pre-) sheaves on $\mathbb 1$ with the category $\mathbb S\mathrm{ets}$ of sets, then we obtain the (final) morphism of topoi
\[
e_{T} : \widetilde T \to \mathbb S\mathrm{ets}.
\]
We have $e_{T*}(E) = \Gamma(T, E)$ and what we want to compute is $ \mathrm H^k(T, E) = \mathrm R^k e_{T*}(E)$.
One may also notice that $e_T^{-1}(S) = \underline S$ is the ``locally constant'' sheaf associated to the set $S$.

We consider now a complete bounded $\delta$-$R$-algebra $A$ and the final morphism of topoi
\[
e_{A} : A_{\mathbb{\Delta}} \to \mathbb S\mathrm{ets}.
\]
The category $\mathbb{\Delta}(A)$ has the final object $A$ so that now $e_{A*}(E) = E_{A}$ is simply the realization of the sheaf $E$ on $A$.
The site $\mathbb{\Delta}(A)$ comes with its sructural ring $\mathcal O_{\mathbb \Delta(A)}$ given by $\mathcal O_{\mathbb \Delta(A)}(B) = B$.
In particular, we have $\mathcal O_{\mathbb \Delta(A)}(A)=A$ and we can promote our final morphism of topoi to a morphism of ringed topoi
\[
e_{A} : \left( A_{\mathbb{\Delta}}, \mathcal O_{\mathbb \Delta(A)}\right) \to ( \mathbb S\mathrm{ets}, A).
\]
We still have $e_{A*}(E) = E_{A}$, but now we also have $e_{A}^*(M) = \mathcal O_{\mathbb \Delta(A)} \otimes_{\underline A} \underline M$ and we will simply write $\mathcal O_{\mathbb \Delta(A)} \otimes_{A} M$: this is the sheaf associated to the presheaf $B \mapsto B \otimes_{A} M$.
When $M$ is finite free, this presheaf is already a sheaf.
More precisely, we have the fundamental result:

%%%%%%%%%%%%%%%%%
\begin{prop} \label{requ}
The functors $\mathrm Re_{A*}$ and $\mathrm Le_{A}^*$ induce an equivalence between locally finite free prismatic crystals on $A$ and finite projective $A$-modules.
\end{prop}

\begin{proof} 
One shows as in proposition \ref{prisheaf} (but it is actually easier here) that $e_{A*}$ and $e_{A}^*$ induce an equivalence between locally finite free prismatic crystals on $A$ and finite projective $A$-modules.
Moreover, if $M$ is finite projective, then it is flat and therefore $\mathrm Le_{A}^*M = e_{A}^*M$.
It only remains to show that $\mathrm R^ie_{A*}E =0$ for $i > 0$ when $E$ is locally finite free.
The question is local and reduces to proving that $\mathrm H^i(A_{\mathbb \Delta}, \mathcal O_{\mathbb \Delta(A)}) =0$ for \emph{any} complete bounded $\delta$-algebra $A$, but this is done in the proof of corollary 3.12 in \cite{BhattScholze22}.
\end{proof}

\begin{rmk}
\begin{enumerate}
\item As a consequence, we get theorem A ($e_{A}^*e_{A*}E =E$) and theorem B ($\mathrm R^ie_{A*}E = 0$ for $i > 0$) for locally finite free prismatic crystals on $A$.
\item In the proof of the proposition, it is actually sufficient to show that the \v Cech cohomology $\check {\mathrm H}^i(A_{\mathbb \Delta}, \mathcal O_{\mathbb \Delta(A)})$ vanishes for any complete bounded $\delta$-algebra $A$ (see \cite[\href{https://stacks.math.columbia.edu/tag/03F9}{Tag 03F9}]{stacks-project}), which is in fact what Bhatt and Scholze do.
\end{enumerate}
\end{rmk}

We consider now the formal scheme $\mathcal X := \mathrm{Spf}(\overline A)$ and the localization functor
\[
j_{A} : \mathbb{\Delta}(A) \to \mathbb{\Delta}(\mathcal X/R).
\]
This is a morphism of prismatic sites which extends to a morphism of topoi 
\[
j_{A} :A_{\mathbb{\Delta}} \to (\mathcal X/R)_{\mathbb{\Delta}}.
\]

%%%%%%%%%%%
\begin{dfn} \label{lindf}
If $M$ is an $A$-module, then the \emph{linearization} of $M$ is the prismatic sheaf $L(M) := j_{A*}e_{A}^*(M)$.
\end{dfn}

\begin{rmk}
If $M$ is a finite projective $A$-module, and we denote by $e_{\mathcal X/R}$ the final morphism on $\mathbb{\Delta}(\mathcal X/R)$, then we have
\[
\Gamma((\mathcal X/R)_{\mathbb{\Delta}}, L(M)) = e_{\mathcal X/R*}L(M) = e_{\mathcal X/R*}j_{A*}e_{A}^* M= e_{A*}e_{A}^* M = M.
\]
It will be essential to prove later in corollary \ref{dervr} a derived version of this statement.
\end{rmk}

We assume from now on that $A$ is endowed with a topologically \'etale coordinate $x$ and we recall that we introduced in definition \ref{prispol} the notion of a ring of prismatic polynomials.
Then, we have:

%%%%%%%%%%%%%%%
\begin{prop} \label{jminus}
If $B$ is a bounded prism on $\mathcal X$ over $R$, then $j_{A}^{-1}(B)$ is representable and we have
\[
j_{A}^{-1}(B) = B [\omega]^{\delta,\wedge}_{d}.
\]
\end{prop}

\begin{proof}
Formally follows from the description of the product in corollary \ref{product}.
\end{proof}

%%%%%%%%%%%%%%%
\begin{cor} \label{jexact}
We have $ \mathrm R^ij_{A*}E = 0$ for $i > 0$ when $E$ is locally finite free.
\end{cor}

\begin{proof}
The sheaf $R^ij_{A*}E$ is associated to the presheaf $B \mapsto \mathrm H^i(B'_{\mathbb \Delta}, E_{|B'})$ with $B' =B [\omega]^{\delta,\wedge}_{d}$.
But it follows from proposition \ref{requ} applied to the case $A = B'$ that $ \mathrm H^i(B'_{\mathbb \Delta}, E_{|B'}) = \mathrm R^ie_{B'*}E_{|B'} = 0$.
\end{proof}

\begin{cor} \label{dervr}
If $M$ is a finite projective $A$-module, then 
\[
\mathrm R\Gamma((\mathcal X/R)_{\mathbb{\Delta}}, L(M)) = M.
\]
\end{cor}

\begin{proof}
It follows from corollary \ref{jexact} and proposition \ref{requ} that

\begin{align*}
\mathrm R\Gamma((\mathcal X/R)_{\mathbb{\Delta}}, L(M)) &=
\mathrm Re_{\mathcal X/R*}L(M) \\&= \mathrm Re_{\mathcal X/R*}j_{A*}e_{A}^* M \\&= \mathrm Re_{\mathcal X/R} \mathrm Rj_{A*}e_{A}^* M \\&= \mathrm R(e_{\mathcal X/R*} j_{A*}) e_{A}^* M \\&= \mathrm Re_{A*}e_{A}^* M \\&= M. \qedhere
\end{align*}
\end{proof}

In the previous corollaries, we only used the representability property, but we now turn to an explicit description of $L(M)$:

%%%%%%%%%%%%%%%%
\begin{lem} \label{compcris}
Let $M$ be a finite projective $A$-module.
\begin{enumerate}
\item If $B$ is a bounded prism on $\mathcal X$ over $R$, then
\[
L(M)_{B} = B [\omega]^{\delta,\wedge}_{d} \otimes_{A}' M.
\]
\item If $B \to C$ is a morphism of bounded prisms on $\mathcal X$ over $R$, then
\[
L(M)_{C} = C \widehat \otimes_{B}L(M)_{B} = C \widehat \otimes^\mathrm{L}_{B}L(M)_{B}.
\]
\end{enumerate}
\end{lem}

\begin{proof}
Follows from proposition \ref{jminus} and lemma \ref{faithpro}.
\end{proof}

%%%%%%%%%
\begin{rmk}
The sheaf $L(M)$ is complete and formally faithfully flat in the sense that $L(M)_B$ is always complete and formally faithfully flat.
This is \emph{not} a crystal \emph{stricto sensu}, but can be called a \emph{complete crystal} taking into account the second statement in lemma \ref{compcris}.
\end{rmk}

The following is a standard intermediate result in a linearization process:

%%%%%%%%%%%
\begin{lem} \label{getout}
If $E$ is a locally finite free prismatic crystal on $\mathcal X$ over $R$ and $M$ is a finite projective $A$-module, then
\[
E \otimes_{\mathcal O_{\mathcal X/R}} L(M) \simeq L(E_{A} \otimes_{A} M).
\]
\end{lem}

\begin{proof}
If $B$ is a bounded prism over $A$, then
\[
(j_{A}^{-1}E)_{B} = B \otimes_{A} (j_{A}^{-1}E)_{A} = B \otimes_{A} E_{A} = (e_{A}^*E_{A})_{B}
\]
because $j_{A}^{-1}E$ is a locally finite free prismatic crystal on $A$.
It follows that $j_{A}^{-1}E =e_{A}^*E_{A}$.
Now, we extend the adjunction map $j_{A}^{-1}j_{A*}e_{A}^*M \to e_{A}^*M$ in order to get

\begin{align*}
j_{A}^{-1}(E \otimes_{\mathcal O_{\mathcal X/R}} j_{A*}e_A^*M) =
&j_{A}^{-1}E \otimes_{\mathcal O_{\mathbb \Delta(A)}} j_{A}^{-1}j_{A*}e_A^*M
\\ &= e_{A}^*E_{A} \otimes_{\mathcal O_{\mathbb \Delta(A)}} j_{A}^{-1}j_{A*}e_A^*M
\\ &\to e_{A}^*E_{A}\otimes_{\mathcal O_{\mathbb \Delta(A)}} e_{A}^*M
\\ &= e_{A}^*(E_{A}\otimes_{A} M).
\end{align*}
By adjunction, we obtain a natural map
\[
E \otimes_{\mathcal O_{\mathcal X/R}} L(M) = E \otimes_{\mathcal O_{\mathcal X/R}} j_{A*}e_A^*M \to j_{A*}(e_{A}^*(E_{A}\otimes_{A} M)) = L(E_{A} \otimes_{A} M).
\]
We want to show that this is an isomorphism.
By additivity, since $M$ is finite projective, we may assume that $M = A$.
It is then sufficient to prove that for any prism $B$ over $\mathcal X/R$,
\[
E_B \otimes_{B} L(A)_B \simeq L(E_{A})_B.
\]
But, since $E$ is a crystal, if we write $B' := B [\omega]^{\delta,\wedge}_{d}$, we have by lemma \ref{compcris}
\[
E_B \otimes_{B} L(A)_B \simeq E_B \otimes_{B} B' \simeq E_{B'} \simeq B' \otimes'_A E_{A} \simeq L(E_{A})_B. \qedhere
\]
\end{proof}

%%%%%%%%%%%
\begin{dfn}
If $B$ is a bounded prism on $\mathcal X$ over $R$, then \emph{comultiplication} is the unique $B$-linear $\delta$-morphism
\[
\Delta_{B} : B [\omega]^{\delta,\wedge}_{d} \to B [\omega]^{\delta,\wedge}_{d} \otimes'_{A}  A[\omega]^{\delta,\wedge}_{d}, \quad \omega \mapsto 1 \otimes' \omega + \omega \otimes' 1.
\]
\end{dfn}

Alternatively, this is the canonical morphism induced by the obvious map
\[
B \to B [\omega]^{\delta,\wedge}_{d} \to B [\omega]^{\delta,\wedge}_{d} \otimes'_{A}  A[\omega]^{\delta,\wedge}_{d}
\]
and the Taylor map
\[
\theta_{B [\omega]^{\delta,\wedge}_{d}} : A \to B [\omega]^{\delta,\wedge}_{d} \otimes'_{A}  A[\omega]^{\delta,\wedge}_{d}.
\]

In the next definition, we should perhaps say hyper-differential operators, but we will again drop the prefix because there is nothing like an non-hyper-differential operator in this theory.

%%%%%%%%%%%%%
\begin{dfn} \label{deldif}
If $M$ and $N$ are two $A$-modules, then a \emph{prismatic differential operator $D : M \to N$} is an $A$-linear map
\[
\widetilde D :  A[\omega]^{\delta,\wedge}_{d} \otimes'_{A} M \to N.
\]
\end{dfn}

We will write $D(s) := \widetilde D(1 \otimes' s)$, but it is important to notice that $\widetilde D$ is not uniquely determined by the map $D$ unless $N$ is $d$-torsion free.

Composition of prismatic differential operators $D : M \to N$ and $E : N \to P$ is obtained as follows
\[
\xymatrix{
 A[\omega]^{\delta,\wedge}_{d} \otimes'_{A} M \ar[d]^-{\Delta_{A} \otimes' \mathrm{Id}_{M}} \ar[rr]^{\widetilde{E \circ D}}&& P
\\
 A[\omega]^{\delta,\wedge}_{d} \otimes_{A}'  A[\omega]^{\delta,\wedge}_{d} \otimes'_{A} M \ar[rr]^-{\mathrm{Id}_{ A[\omega]^{\delta,\wedge}_{d}} \otimes' \widetilde D}
&&  A[\omega]^{\delta,\wedge}_{d} \otimes'_{A} N \ar[u]^-{\widetilde E}.
}
\]
We will denote by
\[
\mathbb\Delta\mathrm{-Diff}(M,N) := \mathrm{Hom}_A(A[\omega]^{\delta,\wedge}_{d} \otimes'_{A} M, N)
\]
the $A$-module of prismatic differential operators from $M$ to $N$, and simply write ${\mathbb\Delta}\mathrm{-Diff}(M)$ when $M=N$.
Multiplication turns ${\mathbb\Delta}\mathrm{-Diff}(M)$ into an $R$-algebra.
If $M$ is an $A$-module, we may then consider the adjunction map (see the proof of proposition 6.16 in \cite{GrosLeStumQuiros20})
\[
M \otimes_A A[\omega]^{\delta,\wedge}_{d} \to \mathrm{Hom}_A(\mathrm{Hom}_A(A[\omega]^{\delta,\wedge}_{d}, A), M) = \mathrm{Hom}_A({\mathbb\Delta}\mathrm{-Diff}(A), M).
\]
If we are given a prismatic stratification $\epsilon$ on $M$, then we can consider the composite map
\[
A[\omega]^{\delta,\wedge}_{d} \otimes'_{A} M \overset \epsilon \simeq  M \otimes_A A[\omega]^{\delta,\wedge}_{d} \to \mathrm{Hom}_A(\mathbb\Delta\mathrm{-Diff}(A), M).
\]
By adjunction again, it provides a morphism of rings
\[
{\mathbb\Delta}\mathrm{-Diff}(A) \to  \mathrm{Hom}_A(A[\omega]^{\delta,\wedge}_{d} \otimes'_{A} M, M) = {\mathbb\Delta}\mathrm{-Diff}(M).
\]
In other words, a prismatic stratification provides an action on $M$ by prismatic differential operators of the ring of prismatic differential operators on $A$.
This is a fully faithful functor.

Now, we define the linearization $L(D)_{B}$ of a prismatic differential operator $D : M \to N$ beween finite projective $A$-modules as follows
\[
\xymatrix{L(M)_{B} \ar[rr]^{L(D)_{B}} \ar@{=}[d]&& L(N)_{B}\ar@{=}[d]\\
B [\omega]^{\delta,\wedge}_{d} \otimes'_{A} M \ar[rd]^{\Delta_{B} \otimes' \mathrm{Id}_{M}} && B [\omega]^{\delta,\wedge}_{d} \otimes'_{A} N \\ & B [\omega]^{\delta,\wedge}_{d} \otimes_{A}'  A[\omega]^{\delta,\wedge}_{d} \otimes'_{A} M \ar[ru]^{\mathrm{Id}_{B} \otimes \widetilde D}
}
\]

%%%%%%%%%%%%%%%%%%%%%
\begin{prop} \label{linfunc}
Linearization provides a functor from the category of finite projective $A$-modules and prismatic differential operators to the category of prismatic sheaves over $A$.
\end{prop}

\begin{proof}
The map $L(D)_{B}$ is clearly $B$-linear and one easily checks that the construction is compatible with completed base extension and composition.
Thanks to lemma \ref{compcris}, we are done.
\end{proof}

\begin{rmks}
\begin{enumerate}
\item 
Note that if $D$ is a prismatic differential operator, then
\[
\mathrm R\Gamma((\mathcal X/R)_{\mathbb{\Delta}}, L(D)) =D.
\]
\item
Again, the results of this section are difficult to use in practice unless we are in a specific situation such as in section \ref{sec6}.
\end{enumerate}
\end{rmks}

%%%%%%%%%%%%%%%%
\section{Prismatic crystals and twisted calculus} \label{sec5}

We recall some constructions from our previous articles \cite{GrosLeStumQuiros20} and \cite{GrosLeStumQuiros21}, and give a local proof that Cartier transform defines an equivalence between prismatic crystals and $q$-crystals.

We endow the local ring $\mathbb Z[q]_{(p,q-1)}$ with the unique $\delta$-structure such that $\delta(q) = 0$ and we denote by $(n)_{q}$ the $q$-analog of an integer $n$.

We let $R$ be a bounded $\delta$-$\mathbb Z[q]_{(p,q-1)}$-algebra (with respect to $p$ modulo $(p)_{q}$).
Then (the image of) $(p)_q$ is a distinguished element in $R$ and $(R, (p)_q)$ will be our base prism (so that $d = (p)_q$ now).
An $R$-module $M$ is systematically endowed with the $(p,q-1)$-adic (or equivalently $(p,(p)_{q})$-adic) topology.

We let $A$ be a complete $R$-algebra with a topologically \'etale coordinate $x$ and we endow it with the unique structure of a $\delta$-$R$-algebra on $A$ such that $\delta(x)=0$.

\subsection*{Prismatic case}

Let us denote by $\sigma$ the unique endomorphism of the $R$-algebra $A$ such that $\sigma(x) = qx$ and $\sigma \equiv \mathrm{Id}_A \mod q-1$.
Recall from proposition 2.4 of \cite{LeStumQuiros18} (applied to $\sigma^p$) that there exists a $q^p$-analog $\mathrm d_{q^p} : A \to \Omega_{A/R,q^p}$ of the universal differential map.
We may then recall from definition 3.1 of \cite{GrosLeStumQuiros21} that a \emph{twisted connection of level $-1$} on an $A$-module $M$ is an $R$-linear map $\nabla : M \to M \otimes \Omega_{A/R,q^p}$ such that
\[
\forall f \in A, \forall s \in M, \quad \nabla(fs) = (p)_{q}\;s \otimes \mathrm d_{q^p}f + \sigma^{p}(f)\nabla(s).
\]
Actually, $\Omega_{A/R,q^p}$ is free on one generator $\mathrm d_{q^p}x$ and we may always write\footnote{We wrote $\partial_q^{\langle 1 \rangle}$ instead of $\partial_{q(-1)}$ in \cite{GrosLeStumQuiros21}, but we'd rather specify the level here.} $\nabla(s) =: \partial_{q(-1)}(s)\mathrm d_{q^p}x$.
The connection is said to be \emph{topologically quasi-nilpotent} if
\[
\forall s \in M, \quad \partial_{q(-1)}^k(s) \to 0.
\]

%%%%%%%%%%%%%%
\begin{thm} \label{equiv1}
The category of locally finite free prismatic crystals $E$ on $\mathcal X := \mathrm{Spf}(A/(p)_{q})$ over $R$ is equivalent to the category of finite projective $A$-modules $M$ endowed with a topologically quasi-nilpotent twisted connection of level $-1$.
\end{thm}

\begin{proof}
We introduced in \cite{GrosLeStumQuiros21}, definition 1.6, the ring $A\langle \omega \rangle_{q(-1)}$ of twisted divided polynomials of level $-1$ and showed in proposition 5.7 of loc.\ cit.\ that its completion is the prismatic envelope of the diagonal.
In other words, we have $ A[\omega]^{\delta,\wedge}_{d}= \widehat{A\langle \omega \rangle}_{q(-1)}$ with the notations of the previous sections.
In this situation, a prismatic stratification is what we called a \emph{twisted hyper-stratification of level $-1$}.
Proposition \ref{equi} then tells us that the category of locally finite free prismatic crystals on $\mathcal X$ is equivalent to the category of projective $A$-modules $M$ endowed with a hyper-stratification of level $-1$.
The conclusion then follows from proposition 3.10 of loc.\ cit.,\ which lets us interpret a twisted hyper-stratification as a twisted connection when the $A$-module is flat and finitely presented (i.e.\ finite projective).
\end{proof}

\begin{rmks}
\begin{enumerate}
\item
As a consequence of this theorem (which also holds in higher dimension) we get a local interpretation of prismatic crystals with respect to $(p)_{q}$ on any smooth formal scheme.
\item
Conversely, the theorem also provides a method to glue $q$-difference equations of negative level, which is a real challenge due to the non-commutative nature of $q$-geometry (see also \cite{Chatzistamatiou20}, corollary 2.3.3 for a $q$-crystalline version).
\end{enumerate}
\end{rmks}

\subsection*{$q$-crystalline case}

We now want to investigate the $q$-crystalline side.
Thus, we assume from now on that $R$ is endowed with a $q$-PD-ideal $\mathfrak r$.
%Note that we will systematically write $\mathrm{Spf}(A) := \mathrm{Spf}(\widehat A)$ when $A$ is an adic ring (this is completely harmless since our adic topology may be defined by a finitely generated ideal).

There exists a $q$-crystalline variant of theorem \ref{lift} which is also due to Bhatt and Scholze, for which we give a short alternative proof:

%%%%%%%%%%%%%%%
\begin{prop}[Bhatt-Scholze] \label{lift2}
If $B$ is a completely flat $\delta$-$R$-algebra and $g \in B$ is completely (faithfully) regular over $R$, then the complex $B[\phi(g)/(p)_{q}]^{\delta}$ is completely (faithfully) flat over $R$.
\end{prop}

\begin{proof}
Note first that $\phi(g)$ is automatically completely (faithfully) regular because $g$ is completely (faithfully) regular and $B$ is completely flat.
Also, it follows from proposition \ref{monoth} in the appendix that $R/(q-1)\widehat \otimes^{\mathrm L}_{R} B$ is discrete because $B$ is completely flat over $R$ (and $R$ has bounded $p^\infty$-torsion).
After complete derived pullback along $R \to R/(q-1)$, we fall back onto lemma \ref{pcase}.
\end{proof}

\begin{rmk}
This proposition is essentially a particular case of the first part of lemma 16.10 in \cite{BhattScholze22}.
In the second part, they also show, with some extra hypothesis, that $B[\phi(g)/(p)_{q}]^{\delta,\wedge}$ is the complete $q$-PD-envelope of $g$ in $B$.
We shall come back to this matter in proposition \ref{lift3}.
\end{rmk}

We assume from now on that $A$ is endowed with a closed $q$-PD-ideal $\mathfrak a$ such that $\mathfrak rA \subset \mathfrak a$.

%%%%%%%%%%%%%%%%%%%
\begin{prop}
If $\mathcal X := \mathrm{Spf}(A/\mathfrak a)$, then $A$ is a covering of (the final object of the topos associated to) $q\mathrm{-CRIS}(\mathcal X/R)$.
\end{prop}

\begin{proof}
This is analogous to the proof of proposition \ref{covering}, but we can briefly recall how it works.
We give ourselves a complete bounded $q$-PD-pair $(B,J)$ over $\mathcal X/R$ and we lift the image of $x$ under $A/\mathfrak a \to B/J$ to some (still written) $x \in B$.
We endow the polynomial ring $B[\xi]$ with $\delta(x+\xi) = 0$.
We let $B' := B[\xi][\phi(\xi)/(p)_{q}]^{\delta, \wedge}$ and denote by $J'$ the closure of $JB'$ in $B'$.
We send $x \in A$ to $x + \xi \in B'$.
This is a morphism of $q$-PD-pairs.
The point now is that the map $B \to B'$ is formally faithfully flat thanks to proposition \ref{lift2}.
\end{proof}

Note that we don't need to know in this proof that $(B',J')$ is indeed the product of $(B, J)$ by $(A, \mathfrak a)$ in the $q$-crystalline site (a fact that will follow from proposition \ref{lift3} below).

We may now consider, as we did above in level $-1$, the $q$-analog $\mathrm d_{q} : A \to \Omega_{A/R,q}$ of the universal differential map and recall from definition 2.8 of \cite{LeStumQuiros18} that a \emph{twisted connection (of level $0$)} on an $A$-module $M$ is an $R$-linear map $\nabla : M \to M \otimes \Omega_{A/R,q}$ such that
\[
\forall f \in A, \forall s \in M, \quad \nabla(fs) =s \otimes \mathrm d_{q}f + \sigma(f)\nabla(s).
\]
Again, $\Omega_{A/R,q}$ is free on $\mathrm d_{q}x$.
We can write $\nabla(s) =: \partial_{q}(s)\mathrm d_{q}x$ and call the connection \emph{topologically quasi-nilpotent} when
\[
\forall s \in M, \quad \partial_{q}^k(s) \to 0.
\]

%%%%%%%%%%%%%
\begin{cor} \label{equiv2}
The category of locally finite free $q$-crystals $E$ on $\mathcal X := \mathrm{Spf}(A/\mathfrak a)$ over $R$ is equivalent to the category of finite projective $A$-modules $M$ endowed with a topologically quasi-nilpotent twisted connection.
\end{cor}

\begin{proof}
This is similar to the proof of proposition \ref{equiv1} using the ring $A\langle \xi \rangle_{q}$ of twisted divided polynomials of level $0$ and theorem 7.3 of \cite{GrosLeStumQuiros20} instead of proposition 5.7 of \cite{GrosLeStumQuiros21}.
\end{proof}

\begin{rmk}
Using the same technics, Andre Chatzistamatiou shows in \cite{Chatzistamatiou20} (theorem 1.3.3 and proposition 2.1.4) that corollary \ref{equiv2} holds in a more general setting.
\end{rmk}

\subsection*{Comparison}

Recall now from definition 4.3 of \cite{GrosLeStumQuiros21} that, if we denote by $\mathrm{MIC}_{q}$ (resp. $\mathrm{MIC}_{q}^{(-1)})$ the category of modules endowed with a twisted connection (resp. a twisted connection of level $-1$), then there exists a \emph{level raising functor}
\[
F^* : \mathrm{MIC}_{q}^{(-1)}(A'/R) \to \mathrm{MIC}_{q}(A/R)
\]
where $A' := R {}_{{}_{\phi}\nwarrow}\!\!\widehat\otimes_{R} A$.

%%%%%%%%%%%%%
\begin{thm} \label{bigdiag}
Let $(R, \mathfrak r) \to (A,\mathfrak a)$ be a morphism of complete bounded $q$-PD-pairs.
Assume that there exists a topologicaly étale coordinate $x$ on $A/R$.
If $\mathcal X := \mathrm{Spf}(A/\mathfrak a)$ and $\mathcal X' := \mathrm{Spf}(A'/(p)_{q})$ with $A' := R {}_{{}_{\phi}\nwarrow}\!\!\widehat\otimes_{R} A$, then there exists a commutative diagram of equivalences
\[
\xymatrix{
\{\mathrm{prismatic}\ \mathrm{crystals}\ \mathrm{on}\ \mathcal X'/R\} \ar[r]^-{C_{\mathcal X/R}^{-1}}_-{\simeq} \ar[d]_-{\simeq} & \{q\mathrm{-crystals}\ \mathrm{on}\ \mathcal X/R\} \ar[d]_-{\simeq} \\
\mathrm{MIC}_{q}^{(-1)}(A'/R) \ar[r]^{F^*}_-{\simeq} & \mathrm{MIC}_{q}^{(0)}(A/R)
}
\]
if we stick to locally finite free crystals, finite projective modules and topologically quasi-nilpotent connections.
\end{thm}

\begin{proof}
We showed in proposition 6.9 of \cite{GrosLeStumQuiros21} that there exists such a commutative diagram. We proved in theorem 4.8 of loc.\ cit.\ that the bottom map is an equivalence and we have proved in proposition \ref{equiv1} and corollary \ref{equiv2} that both vertical maps are equivalences.
\end{proof}

\begin{rmks}
\begin{enumerate}
\item
Our proof relies on Frobenius descent for twisted modules endowed with a connection and is independent of the results in \cite{Li21} or in \cite{Chatzistamatiou20}.
In particular, we recover theorem \ref{Liplus} in the particular setting of this section.
\item Theorem \ref{bigdiag} can be seen as a $q$-deformation of proposition 9.17 in \cite{Xu19}.
\end{enumerate}
\end{rmks}

%%%%%%%%%%%
\section{Cohomology} \label{sec6}

We will show that prismatic cohomology (resp.\ $q$-crystaline cohomology) agrees with twisted de Rham cohomology of level $-1$ (resp.\ of level $0$) and prove a comparison theorem in de Rham cohomology.

We keep the assumptions and notations of the previous section.
More precisely, we consider a base prism of the form $(R, (p)_q)$ where $R$ is a $\delta$-$\mathbb Z[q]_{(p,q-1)}$-algebra and $\delta(q) = 0$.
We let $A$ be a complete $R$-algebra with a topologically \'etale coordinate $x$ (and we set $\delta(x)=0$).

%%%%%%%%%%
\subsection*{Prismatic case}

If $M$ is an $A$-module endowed with a twisted connection $\nabla : M \to M \otimes_{A} \Omega_{A/R,q^p}$ of level $-1$, then its \emph{de Rham cohomology} $\mathrm H^i_{\mathrm{dR},q(-1)}(M)$ is simply the cohomology of the de Rham complex

\begin{align} \label{dRminun}
\left[M \overset \nabla \to M \otimes_{A} \Omega_{A/R,q^p}\right].
\end{align}
The twisted connection $\nabla$ provides a \emph{twisted hyper-differential operator of level $-1$ (and order $1$)} in the sense that $\nabla$ extends canonically to
\[
\xymatrix@R=0cm{
\widehat{A\langle \omega \rangle}_{q(-1)} \otimes'_{A }M \ar[rr]^-{\widetilde \nabla}& & M \otimes \Omega_{A/R,q^p}
\\ \omega \otimes s \ar@{|->}[rr] && s \otimes \mathrm dx + (q-1)x\nabla(s).}
\]
We know from proposition 5.7 in \cite{GrosLeStumQuiros21} that $ A[\omega]^{\delta,\wedge}_{d} = \widehat{A\langle \omega \rangle}_{q(-1)}$ is the completed ring of twisted divided polynomials of level $-1$.
It follows that a prismatic differential operator as in definition \ref{deldif} is then the same thing as a \emph{twisted hyper-differential operator of level $-1$}.

The de Rham complex \eqref{dRminun} is endowed with a semilinear frobenius $\phi$ induced by
\[
\phi : \Omega_{A/R,q^p} \to \Omega_{A/R,q^p}, \quad \mathrm d_{p^q}x \mapsto (p)_{q}x^{p-1}\mathrm d_{p^q}x.
\]
%In particular, if we denote by $\partial d$ the twisted derivation of we will have $\partial(\phi(s)) = (p)_qx^{p-1}\phi(\partial s))$.
This formula comes from the explicit description
\[
\phi(\omega) = \sum_{k=1}^{p} {p-1 \choose k-1}_{q^p} (p)_q^kx^{p-k} \omega^{\{p\}}
\]
 of the absolute frobenius on $A\langle \omega \rangle_{q(-1)}$ that can be deduced from the computations in the proof of proposition 1.12 of \cite{GrosLeStumQuiros21}.

%%%%%%%%%%%%%%
\begin{prop}[Prismatic Poincar\'e lemma]
If $E$ is a locally finite free prismatic crystal on $\mathcal X := \mathrm{Spf}(A/(p)_q)$ over $R$ and $\nabla : E_{A} \to E_{A} \otimes_{A} \Omega_{A/R,q^p}$ denotes the corresponding twisted connection of level $-1$, then the sequence
\[
0 \to E \to L(E_{A}) \overset{L(\nabla)} \to L(E_{A} \otimes_{A} \Omega_{A/R,q^p}) \to 0
\]
is exact.
\end{prop}

\begin{proof}
We have to show that, for any bounded prism $B$ on $\mathcal X/R$, up to a formally faithfully flat map, the sequence
\[
0 \to E_{B} \to L(E_{A})_{B} \overset{L(\nabla)} \to L(E_{A} \otimes_{A} \Omega_{A/R,q^p})_{B} \to 0
\]
is exact.
Actually, since $A$ is a covering of the site, we may assume that there exists a $\delta$-morphism $A \to B$.
Moreover, thanks to assertion 2) of lemma \ref{compcris}, we may assume that $B=A$. 
Finally, using lemma \ref{getout}, we can assume that $E = \mathcal O_{\mathcal X/R}$ is the structural sheaf.
We are therefore reduced to prove the exactness of the augmented de Rham complex
\[
0 \to A \to \widehat{A\langle \omega \rangle}_{q(-1)} \overset{L(\mathrm d)_{A}} \to \widehat{A\langle \omega \rangle}_{q(-1)} \otimes_{A} \Omega_{A/R,q^p} \to 0.
\]
Hence, it is sufficient to show that
\[
\forall k \in \mathbb N, \quad L(\mathrm d)_{A}(\omega^{\{k+1\}}) = \omega^{\{k\}} \otimes \mathrm dx.
\]
But this follows immediately from the standard formulas
\[
\Delta_{A}(\omega^{\{k\}}) = \sum_{i+j=k} \omega^{\{i\}} \otimes \omega^{\{j\}}, \quad \widetilde d(\omega) = 1 \quad \mathrm{and} \quad \widetilde d(\omega^{\{k\}}) = 0\ \mathrm{for}\ k > 1
\]
(since $d$ is a twisted differential operator of order $1$).
\end{proof}

The Poincar\'e lemma provides an isomorphism on cohomology:

%%%%%%%%%%%
\begin{thm} \label{prisdR}
If $E$ is a locally finite free prismatic crystal on $\mathcal X := \mathrm{Spf}(A/(p)_q)$ over $R$, then
\[
\mathrm R \Gamma((\mathcal X/R)_{\mathbb{\Delta}}, E) \simeq \left[E_{A} \overset \nabla \to E_{A} \otimes_{A} \Omega_{A/R,q^p}\right].
\]
\end{thm}

\begin{proof} With the notations of section \ref{sec4}, we have

\begin{align*}
\mathrm  R \Gamma((\mathcal X/R)_{\mathbb{\Delta}}, E) &\simeq \mathrm R e_{\mathcal X/R*}E
\\& \simeq \mathrm R e_{\mathcal X/R*} \left[L(E_{A}) \overset{L(\nabla)} \to L(E_{A} \otimes_{A} \Omega_{A/R,q^p})\right]
\\ & \simeq \left[E_{A} \overset \nabla \to E_{A} \otimes_{A} \Omega_{A/R,q^p}\right].\qedhere
\end{align*}
\end{proof}

As a consequence of this theorem, we see that
\[
\forall i \in \mathbb N, \quad \mathrm H^i((\mathcal X/R)_{\mathbb{\Delta}}, E) \simeq \mathrm H^i_{\mathrm{dR},q(-1)}(E_{A}).
\]

%%%%%%%%%%
\begin{rmks}
\begin{enumerate}
\item
From theorem \ref{prisdR}, we obtain in this setting a canonical \emph{Hodge decomposition}
\[
\mathrm R \Gamma((\mathcal X/R)_{\mathbb{\Delta}}, \overline {\mathcal O}) \simeq \overline A \oplus \Omega_{\overline A/\overline R}[-1],
\]
in which the bar denotes reduction modulo $(p)_q$ and $ \Omega_{\overline A/\overline R}$ is the usual module of differentials.
This is an incarnation in the derived category of theorem 6.3 of \cite{BhattScholze22}: the maps induced in cohomology can easily be identified with the isomorphisms
\[
\mathrm H^k((\mathcal X/R)_{\mathbb{\Delta}}, \overline {\mathcal O}) \simeq \Omega^k_{\overline A/\overline R}\{k\}
\]
of loc.\ cit.
We also obtain an identification of the usual differential $\mathrm d : \overline A \to \Omega_{\overline A/\overline R}$ with the connection that comes by functoriality from the Bockstein exact sequence
\[
0 \longrightarrow \overline R \overset {(p)_{q}} \longrightarrow R/(p)_{q}^2 \longrightarrow \overline R \longrightarrow 0.
\]
\item This Hodge decomposition is also obtained by Tian (remark 4.16 in \cite{Tian21}).
\item
The isomorphism of theorem \ref{prisdR} is functorial: it sends the semilinear frobenius endomorphism $\phi$ of $\mathrm R \Gamma((\mathcal X/R)_{\mathbb{\Delta}}, E)$ to the semilinear endomorphism $\phi$ of $[E_{A} \overset \nabla \to E_{A} \otimes_{A} \Omega_{A/R,q^p}]$.
% where $\phi(\mathrm d_{q^p}x) = (p)_{q^p}x^{p-1}\mathrm d_{q^p}x$.
\end{enumerate}
\end{rmks}

%%%%%%%%%%%%%%%%%%%%
\subsection*{$q$-crystalline case}

We will now consider the $q$-crystalline case and therefore assume that $R$ is endowed with a $q$-PD-ideal $\mathfrak r$.
We will assume that $q-1 \in \mathfrak r$.
We first recall the following:

%%%%%%%%%%%%
\begin{dfn}
The \emph{(complete) $q$-PD-envelope} of an ideal $J$ in a $\delta$-$R$-algebra $B$ is a (complete) $q$-PD-pair $(C,K)$ over $R$ which which is universal for morphisms $B \to C$ sending $J$ into $K$.
\end{dfn}

We may also say that $C$ is the (complete) $q$-PD-envelope, meaning that there exists such a $K$, which is then unique.

Let us now reproduce the following from lemma 16.10 of \cite{BhattScholze22}:

%%%%%%%%%%%%%%%
\begin{prop}[Bhatt-Scholze] \label{lift3}
Assume that $B$ is a formally flat bounded $\delta$-$R$-algebra such that $B/(q-1)B$ is $p$-torsion free.
If $g \in B$ is completely regular over $R$, then $B' := B[\phi(g)/(p)_{q}]^{\delta,\wedge}$ is the complete $q$-PD-envelope of $J:= \mathfrak rB + (g)$.
Moreover, if $J'$ denotes the $q$-PD-ideal of $B'$, then $B/J\simeq B'/J'$.
\end{prop}

\begin{proof}
Recall from proposition \ref{lift2} that $B'$, which is derived complete, is also completely flat over $R$.
Since $R$ is bounded, this implies that $B'$ is discrete, classically complete and bounded (theorem \ref{compform} in the appendix).
Also, when $q=1$, the ring
\[
B' = B[\phi(g)/p]^{\delta,\wedge} \simeq B^{\mathrm{PD},\wedge}
\]
is the completed PD-envelope of $J$ and the result therefore holds.
In particular, if $J''$ is another closed PD-ideal such that $JB' \subset J'' \subset J'$, then necessarily $J''=J'$.
We now do the general case and indicate (only in this proof) reduction modulo $q-1$ with a bar.
We denote by $J'$ be the kernel of the morphism
\[
B' \to \overline B' \to \overline B/\overline J \simeq B/J.
\]
It is sufficient to show that $J'$ is a $q$-PD-ideal.
The assertion will then follow from the universal property of $B'$.

Since $\overline B$ is $p$-torsion free (and therefore a bounded prism), it follows from the case $q=1$ that $(\overline {B'}, \overline {J'})$ is the PD-envelope of $\overline J$.
Let us denote by $J''$ the intersection of $J'$ with the Nygaard ideal $\phi^{-1}((p)_qB')$.
$J''$ is clearly a closed $q$-PD-ideal containing $JB'$ and it is therefore sufficient to show that $J''=J'$.
But it is then also true that $\overline {J''}$ is a closed PD-ideal of $\overline B'$ containing $\overline J\overline {B'}$ and it follows from the case $q=1$ that $\overline {J''} = \overline {J'}$.
It follows from the fact that $q - 1$ belongs to the Nygaard ideal that $J''=J'$.
\end{proof}

Let us assume now that $A$ is endowed with a closed $q$-PD-ideal $\mathfrak a$ such that $\mathfrak rA \subset \mathfrak a$.
If $M$ is an $A$-module endowed with a twisted connection $\nabla : M \to M \otimes_{A} \Omega_{A/R,q}$, then its de Rham cohomology $\mathrm H^i_{\mathrm{dR},q}(M)$ is the cohomology of the de Rham complex
\[
\left[M \overset \nabla \to M \otimes_{A} \Omega_{A/R,q}\right].
\]

From now on, we add the following technical condition for a $q$-PD-pair $(B,J)$:

$\quad (\star) \quad$ $B/(q-1)B$ is $p$-torsion free.

Note that all the previous results are still valid with this extra assumption.
Based on proposition \ref{lift3}, we can then apply the prismatic method to the $q$-crystalline situation and show that:

%%%%%%%%%%%%
\begin{thm} \label{qcrdR}
If $E$ is a locally finite free $q$-crystal on $\mathcal X := \mathrm{Spf}(A/\mathfrak a)$ over $R$, then
\[
\mathrm R\Gamma((\mathcal X/R)_{q\mathrm{-CRIS}}, E) \simeq \left[E_A \overset \nabla \to E_A \otimes_{A} \Omega_{A/R,q}\right]. \qed
\]
\end{thm}

As a consequence, we will have
\[
\forall i \in \mathbb N, \quad \mathrm H^i((\mathcal X/R)_{q\mathrm{-CRIS}}, E) \simeq \mathrm H^i_{\mathrm{dR},q}(E_{A}).
\]

\subsection*{Comparison}

We finish by proving that raising level provides an isomorphism on de Rham complexes.
We write $A' := R {}_{{}_{\phi}\nwarrow}\!\!\widehat\otimes_{R} A$ and we consider $A$ as an $A'$-algebra via the relative Frobenius $F : A' \to A$.

%%%%%%%%%%%
\begin{thm} \label{lastthm}
If $M'$ is a complete formally flat $A'$-module endowed with a topologically quasi-nilpotent twisted connection of level $-1$ and $M := A \otimes_{A'} M'$, then frobenius induces a quasi-isomorphism
\[
[M' \to M' \otimes \Omega_{A'/R,q^p}] \simeq [M \to M \otimes \Omega_{A/R,q}].
\]
\end{thm}

\begin{proof}
Our map is induced by frobenius
\[
F : M' \to M := A \otimes_{A'} M', \quad s \mapsto 1 \otimes s,
\]
but we have to be careful because the frobenius on differentials is induced by the divided Frobenius
\[
[F] : \Omega_{A'/R,q^p} \to \Omega_{A/R,q}, \quad \mathrm dx' \mapsto x^{p-1}\mathrm d_qx,
\]
where we write $x':=1 {}_{{}_{\phi}\nwarrow}\!\!\widehat\otimes x \in A'$.
Let us denote by $\theta, \theta'$ the corresponding twisted derivations of level $0$ and $-1$ respectively and set
\[
[F] : M' \to M , \quad s \mapsto x^{p-1}F(s).
\]
Let us also recall that
\[
\forall s \in M', \quad \theta(1 \otimes s) = x^{p-1} \otimes \theta'(s).
\]

Thus, we have to prove that the vertical map of complexes
\[
\xymatrix{M' \ar[r]^{\theta'} \ar[d]^{F}& M' \ar[d]^{[F]} \\ M \ar[r]^\theta & M}
\]
is a quasi-isomorphism.
Since $A$ is free on $A'$ with generators $1, x, \ldots, x^{p-1}$, we have
\[
M := A \otimes_{A'} M' \simeq \oplus_{k=0}^{p-1}\ A'x^k \otimes_{A'} M' = \oplus_{k=0}^{p-1}\ x^kM'.
\]
as an $A'$-module.
On the one hand,
\[
\theta(s) = \theta (1 \otimes s) = x^{p-1} \otimes \theta'(s) = x^{p-1} \theta'(s),
\]
and there exists therefore a vertical isomorphism of complexes
\[
\xymatrix{M' \ar[r]^{\theta'} \ar[d]^{F} & M' \ar[d]^{[F]} \\ M' \ar[r]^\theta & x^{p-1} M'. }
\]
On the other hand, we have for all $1 \leq k < p$,

\begin{align*}
\theta(x^ks)&= \theta(x^k \otimes s)
\\ &= x^k \theta(1 \otimes s) + (k)_qx^{k-1} \otimes s
\\ &= x^kx^{p-1} \otimes \theta'(s) + x^{k-1} \otimes (k)_q s
\\ &= x^{k-1} \otimes x'\theta'(s) + x^{k-1} \otimes (k)_q s
\\ &= x^{k-1}(x'\theta'(s)+ (k)_qs).
\end{align*}

Our assertion therefore reduces to showing that, for $1 \leq  k < p$, the map
\[
M' \to M', \quad s \mapsto x'\theta'(s)+ (k)_qs
\]
is bijective.
We may then invoke the derived Nakayama lemma (see proposition \ref{Nak}) and assume that $q=1$ and $p=0$, in which case we fall back onto classical Cartier isomorphism (corollary 2.28 of \cite{OgusVologodsky07}).
We could also only assume that $(p)_q=0$ and rely on corollary 8.10 of \cite{GrosLeStumQuiros19}, which is more in the spirit of this article.
\end{proof}

\begin{rmks}
\begin{enumerate}
\item 
As a consequence of this theorem, there exist isomorphisms
\[
\forall k \in \mathbb N, \quad \mathrm H^k_{\mathrm{dR},q(-1)}(M') \simeq \mathrm H^k_{\mathrm{dR},q}(M).
\]
\item
In the case $q = 1$, theorem \ref{lastthm} is not new.
In this situation, it is in fact just an improvement, already noticed by Arthur Ogus, of a result of Atsushi Shiho (\cite{Shiho15}, theorem 4.4).
\item When $M'$ is finitely presented and the connection is topologically quasi-nilpotent, one can recover this result from theorem \ref{Liplus} thanks to theorems \ref{prisdR} and \ref{qcrdR}.
\item Conversely, theorem \ref{lastthm} provides us with a new proof of theorem \ref{Liplus} in our particular setting.
\end{enumerate}
\end{rmks}

\begin{cor} \label{imphi}
If $E'$ is a locally finite free prismatic crystal on $\mathcal X' := \mathrm{Spf}(A'/(p)_{q})$ over $R$, then frobenius induces an isomorphism
\[
\mathrm R\Gamma\left((\mathcal X'/R)_{\mathbb\Delta}, E'\right)\left[\frac 1{(p)_q}\right] \simeq \mathrm R\Gamma\left((\mathcal X/R)_{\mathbb\Delta}, F^*E'\right)\left[\frac 1{(p)_q}\right].
\]

\end{cor}

 \begin{proof}
If we let $M' := E'_{A'}$ and $M := A \otimes_{A'} M'$, then we already know that

\begin{align*}
\mathrm R\Gamma\left(\mathcal X'/R)_{\mathbb\Delta}, E'\right) & \simeq \left[M' \to M' \otimes_{A'} \Omega_{A'/R,q^p}\right]
\\ & \simeq  [M \to M \otimes \Omega_{A/R,q}]
\end{align*}
(with a twisted connection of level $0$).
But we also have $(F^*E')_A = A \otimes_{A'} M' = M$ and therefore
\[
\mathrm R\Gamma\left(\mathcal X/R)_{\mathbb\Delta}, F^*E'\right) = \left[M \to M \otimes_{A} \Omega_{A/R,q^p}\right]
\]
(with a twisted connection of level $1$).
Now, the canonical map
\[
[M \to M \otimes \Omega_{A/R,q}] \to [M \to M \otimes \Omega_{A/R,q^p}]
\]
is the identity on $M$ and the map induced on differentials is the ``Verschiebung''
\[
\Omega_{A/R,q} \longrightarrow \Omega_{A/R,q^p}, \quad \mathrm d_qx \mapsto (p)_{q} \mathrm d_{q^p} x
\]
(see section 1 of \cite{GrosLeStumQuiros21}).
\end{proof}

\begin{rmk}
In the case of the trivial crystal, corollary \ref{imphi} is due to Bhatt and Scholze (see theorem 1.8 (6) in the introduction of \cite{BhattScholze22}).
\end{rmk}

There exists a variant to this corollary:

%%%%%%%%
\begin{cor}
If $E$ is a locally finite free prismatic crystal on $\mathcal X := \mathrm{Spf}(A/(p)_q)$ over $R$, then frobenius induces an isomorphism
\[
R {}_{{}_{\phi}\nwarrow}\!\!\widehat\otimes^\mathrm L_{R} \mathrm R\Gamma\left((\mathcal X/R)_{\mathbb\Delta}, E\right)\left[\frac 1{(p)_{q^p}}\right] \simeq \mathrm R\Gamma\left((\mathcal X/R)_{\mathbb\Delta}, \phi^*E\right)\left[\frac 1{(p)_{q^p}}\right].
\]
\end{cor}

\begin{proof}
If we pull back $E$ along the frobenius morphism $\phi : (R,(p)_q) \to (R, (p)_{q^p})$, we get some $E'$ and we have

\begin{align*}
R {}_{{}_{\phi}\nwarrow}\!\!\widehat\otimes^\mathrm L_{R} \mathrm R\Gamma\left(\mathcal X/R)_{\mathbb\Delta}, E\right)
&\simeq R {}_{{}_{\phi}\nwarrow}\!\!\widehat\otimes^\mathrm L_{R}  \left[E_{A} \overset \nabla \to E_{A} \otimes_{A} \Omega_{A/R,q^p}\right]
\\ & \simeq \left[ E'_{A'} \overset \nabla \to E'_{A'} \otimes_{A'} \Omega_{A'/R',q^{p^2}}\right]
\\ & \simeq \mathrm R\Gamma\left((\mathcal X'/R)_{\mathbb\Delta}, E'\right),
\end{align*}
and
\[
\mathrm R\Gamma\left((\mathcal X/R)_{\mathbb\Delta}, F^*E'\right) \simeq \mathrm R\Gamma\left((\mathcal X/R)_{\mathbb\Delta}, \phi^*E\right).
\]
We may then apply corollary \ref{imphi}.
\end{proof}

%%%%%%%%%%%%%%%
\setcounter{section}{0}
\renewcommand{\theHsection}{\Alph{section}}
\appendix
\section{Appendix}

We gather some results found here and there about complete flatness and derived completion that we use in the core of the article.
We however do not discuss the simplicial approach.
Our aim is mostly to explain how complete flatness and derived completion boil down in practice to the discrete notions of classical completion and formal flatness.

\textbf{Caveat:} In contrast with the rest of the article, we use in this appendix the \emph{hat} notation to denote derived completion (and not classical completion).

We let $A$ be a (not necessarily complete) adic ring.
All $A$-modules are endowed with their adic topology.

At some point, we will assume that $A$ has a finitely generated ideal of definition.

%%%%%%%%%%%%%%%%
\subsection{Complete flatness} \label{Append1}

When we say that $M^\bullet$ is a \emph{complex of $A$-modules}, we mean that it is an object of the derived category $\mathrm D(A)$ of the category of $A$-modules.
We consider the latter as a full subcategory of $\mathrm D(A)$.

%%%%%%%%%
\begin{dfn}
\begin{enumerate}
\item
A complex of $A$-modules $M^\bullet$ is \emph{discrete} if $\mathrm H^{k}(M^\bullet) = 0$ whenever $k \neq 0$.
\item
A complex of $A$-modules $M^\bullet$ is \emph{discrete (faithfully) flat} (over $A$) if it is \emph{discrete} and $\mathrm H^{0}(M^\bullet)$ is (faithfully) flat over $A$.
\end{enumerate}
\end{dfn}

\begin{rmks}
\begin{enumerate}
\item
The inclusion of $A$-modules into complexes of $A$-modules induces an equivalence between $A$-modules (resp.\ (faithfully) flat $A$-modules) and discrete (resp.\ discrete  (faithfully) flat) complexes of $A$-modules.
An inverse is given by the functor $M^\bullet \mapsto \mathrm H^0(M^\bullet)$.
\item Since we will have to move back and forth between $A$-modules and complexes of $A$-modules, we also recall the following:
\begin{enumerate}\item
If $L^\bullet \to M^\bullet \to N^\bullet \to$ is a distinguished triangle and two of the complexes are discrete, then so is the third.
\item A sequence $0 \to L \to M \to N \to 0$ of $A$-modules is exact if and only if the triangle $L \to M \to N \to$ is distinguished.
\end{enumerate}
\item A complex of $A$-modules $M^\bullet$ is discrete (resp.\ discrete  (faithfully) flat) if and only if $M^\bullet \simeq \mathrm H^{0}(M^\bullet)$ (resp.\ $M^\bullet \simeq \mathrm H^{0}(M^\bullet)$ and $\mathrm H^{0}(M^\bullet)$ is  (faithfully) flat).
\item A complex of $A$-modules $M^\bullet$ is discrete flat if and only if for any $A$-module $N$, the complex $M^\bullet \otimes^{\mathrm L}_{A} N$ is discrete and then
\[
M^\bullet \otimes^{\mathrm L}_{A} N \simeq \mathrm H^0(M^\bullet \otimes^{\mathrm L}_{A} N ) \simeq \mathrm H^0(M^\bullet) \otimes_{A} N.
\]
\end{enumerate}
\end{rmks}

%%%%%%%%%%%
\begin{lem} \label{power}
Let $M^\bullet$ be a complex of $A$-modules and $I, J \subset A$ two ideals.
If $M^\bullet \otimes_{A}^{\mathrm L} A/I$ is discrete flat over $A/I$ and $I^{n+1} \subset J$, then $M^\bullet \otimes_{A}^{\mathrm L} A/J$ is discrete flat over $A/J$.
\end{lem}

\begin{proof}
It is sufficient to do the cases $I=0$ and $J=I^2 =0$.
\begin{enumerate}
\item Assume $I=0$.
Then $M^\bullet$ is discrete  (faithfully) flat over $A$ so that
\[
M^\bullet \otimes_{A}^{\mathrm L} A/J \simeq \mathrm H^0(M^\bullet \otimes_{A}^{\mathrm L} A/J) \simeq \mathrm H^{0}(M^\bullet) \otimes_{A} A/J
\]
and we know that scalar extension preserves (faithful) flatness.
\item Assume $J=I^2 =0$.
If $N$ is any $A$-module, then the short exact sequence $0 \mapsto IN \to N \to N/IN \to 0$ provides a distinguished triangle
\[
M^\bullet \otimes_{A}^{\mathrm L} IN \to M^\bullet \otimes_{A}^{\mathrm L} N \to M^\bullet \otimes^{\mathrm L}_{A} N/IN \to.
\]
Of course, $N/IN$ is an $A/I$-module, but since $I^2 = 0$, $IN$ is also an $A/I$-module and we may rewrite our triangle:
\[
(M^\bullet \otimes_{A}^{\mathrm L} A/I) \otimes_{A/I} IN \to M^\bullet \otimes_{A}^{\mathrm L} N \to (M^\bullet \otimes_{A}^{\mathrm L} A/I) \otimes^{\mathrm L}_{A/I} N/IN \to.
\]
Since $M^\bullet \otimes_{A}^{\mathrm L} A/I$ is discrete flat over $A/I$, we see that both sides are discrete and $M^\bullet \otimes_{A}^{\mathrm L} N$ must also be discrete. \qedhere
\end{enumerate}
\end{proof}

%%%%%%%%%
\begin{dfn} \label{compfl}
A complex of $A$-modules $M^\bullet$ is \emph{completely (faithfully) flat} if $M^\bullet \otimes_{A}^{\mathrm L} A/I$ is discrete (faithfully) flat over $A/I$ for \emph{some} ideal of definition $I$.
\end{dfn}

\textbf{Warning} : In this definition, the complex $M^\bullet$ itself need not be discrete.

%The following properties are immediate consequences of lemma \ref{power}:

%%%%%%%%%%%%
\begin{rmks}
\newcounter{mycount}
\begin{enumerate}
\item 
If a complex of $A$-modules $M^\bullet$ is completely  (faithfully)  flat, then $M^\bullet \otimes_{A}^{\mathrm L} A/I$ is discrete  (faithfully) flat over $A/I$ for \emph{any} ideal of definition $I$.
\item A complex of $A$-modules $M^\bullet$ is discrete  (faithfully) flat if and only if it is completely  (faithfully)  flat for the discrete topology.
\item Complete  (faithful) flatness is inherited by any coarser adic topology (bigger ideal of definition).
\item If a complex of $A$-modules $M^\bullet$ is discrete  (faithfully)  flat, then it is completely  (faithfully)  flat.
\setcounter{mycount}{\theenumi}
\end{enumerate}
We now let $I$ be an ideal of definition for $A$.
\begin{enumerate}
\setcounter{enumi}{\themycount}
\item If $J \subset I$, then $M^\bullet$ is completely  (faithfully) flat if and only if $M^\bullet \otimes^{\mathrm L}_{A} A/J$ is completely  (faithfully)  flat for the $I/J$-adic topology.
\item \label{rmk}
\begin{enumerate}
\item A complex of $A$-modules $M^\bullet$ is completely  (faithfully) flat if and only if $\mathrm H^0(M^\bullet \otimes_{A}^{\mathrm L} A/I)$ is  (faithfully)  flat over $A/I$ and $\mathrm H^k(M^\bullet \otimes_{A}^{\mathrm L} A/I) = 0$ for $k \neq 0$.
\item \label{rmkplus}An $A$-module $M$ is completely flat if and only if $M/IM$ is flat over $A/I$ and $M \otimes^{\mathrm L}_{A} A/I = M \otimes_{A} A/I$ (or equivalently $\mathrm{Tor}_{k}^A(M, A/I) = 0$ for $k \neq 0$).
\end{enumerate}
\item
\begin{enumerate}
\item If a complex of $A$-modules $M^\bullet$ is completely flat, then
\[
M^\bullet \otimes_{A}^{\mathrm L} A/I \simeq \mathrm H^0(M^\bullet \otimes_{A}^{\mathrm L} A/I).
\]
\item 
If an $A$-module $M$ is completely flat, then
\[
M \otimes_{A}^{\mathrm L} A/I \simeq M/I M.
\]
\end{enumerate}
\item
\begin{enumerate}
\item A complex of $A$-modules $M^\bullet$ is completely flat if and only if for all $A/I$-module $N$, the complex $M^\bullet \otimes^{\mathrm L}_{A} N$ is discrete.
\item An $A$-module $M$ is completely flat if and only if for all $A/I$-module $N$, $M \otimes^{\mathrm L}_{A} N = M \otimes_{A} N$ (i.e. $\mathrm{Tor}_{k}^A(M,N) = 0$ for $k \neq 0$).
\end{enumerate}
\end{enumerate}
\end{rmks}

%%%%%%%%%
\begin{dfn}
An adic morphism of adic rings $u : A \to B$ is said to be \emph{completely (faithfully) flat} if $B$ is completely (faithfully) flat as an $A$-module.
\end{dfn}

\begin{rmks}
Let $u : A \to B$ is an adic morphism of adic rings.
\begin{enumerate}
\item
If $M^\bullet$ is a completely (faithfully) flat complex of $A$-modules, then $B \otimes^{\mathrm L}_A M^\bullet$ is a completely (faithfully) flat complex of $B$-modules.
\item 
Conversely, if $u$ is completely faithfully flat and $B \otimes^{\mathrm L}_A M^\bullet$ is a completely (faithfully) flat, then $M^\bullet$ was already completely (faithfully) flat.
\end{enumerate}
\end{rmks}

%%%%%%%%%%%%%
\begin{dfn} \label{formflat}
\begin{enumerate}
\item
An $A$-module $M$ is \emph{formally (faithfully) flat}\footnote{Also called \emph{adically flat}.} if $M/IM$ is (faithfully) flat over $A/I$ for \emph{any} ideal of definition $I$.
\item
A complex of $A$-modules $M^\bullet$ is \emph{formally (faithfully) flat} if it is \emph{discrete} and $\mathrm H^0(M^\bullet)$ is formally (faithfully) flat.
\end{enumerate}
\end{dfn}

%%%%%%%%%%%%%
\begin{rmks}
\begin{enumerate}
\item
If $I$ is an ideal of definition of $A$, then an $A$-module $M$ is formally (faithfully) flat if and only if $M/I^{n+1}M$ is (faithfully) flat over $A/I^{n+1}$ for all $n \in \mathbb N$.
\item 
It is clearly not sufficient that $M/IM$ is (faithfully) flat over $A/I$ for \emph{some} ideal of definition $I$.
\item
A (faithfully) flat $A$-module $M$ is automatically completely (faithfully) flat and a completely (faithfully) flat $A$-module $M$ is automatically formally (faithfully) flat.
\end{enumerate}
\end{rmks}

%%%%%%%%%%%%%
\subsection{Derived completeness} \label{Append2}

In order to go further, we need to recall some results about derived completeness (we refer the reader to \cite[\href{https://stacks.math.columbia.edu/tag/091N}{Tag 091N} %{Tag 0BKF}
]{stacks-project} for the details).

%%%%%%%%%%%%
\begin{dfn} \label{defcomp}
\begin{enumerate}
\item
An $A$-module $M$ is \emph{classically complete}\footnote{Also called \emph{adically complete.}} if
\[
M \simeq \varprojlim_{n \in \mathbb N} M/I^{n}
\]
for some (or any) ideal of definition.
\item
A complex of $A$-modules $M^\bullet$ is \emph{classically complete} if it is \emph{discrete} and $\mathrm H^0(M^\bullet)$ is classically complete.
\item
A complex of $A$-modules $M^\bullet$ is \emph{derived complete} if

\begin{align} \label{derc}
\forall f \in I, \quad \mathrm R\varprojlim_{f} M^\bullet = 0
\end{align}
for \emph{some} ideal of definition $I$ of $A$.
\end{enumerate}
\end{dfn}

%%%%%%%%%%%%
\begin{rmks}
\begin{enumerate}
\item If a complex of $A$-modules $M^\bullet$ is derived complete, then \eqref{derc} holds for any ideal of definition $I$ of $A$ (and even for $\sqrt{I}$).
\item Derived completeness may be checked on generators of $I$.
\item Derived completeness is inherited by any finer adic topology (smaller ideal of definition).
\item If a complex of $A$-modules $M^\bullet$ is classically complete, then $M^\bullet$ is automatically derived complete.
\end{enumerate}
\end{rmks}

Derived completeness is a remarkably stable notion:

%%%%%%%%%%%%%%%%%%%%%%
\begin{prop} \label{stb}
\begin{enumerate}
\item Derived complete complexes of $A$-modules form a saturated\footnote{That is, stable under direct factor.} full triangulated subcategory of the category of all complexes of $A$-modules 
\item A complex of $A$-modules is derived complete if and only if it has derived complete cohomology.
\item Derived complete $A$-modules form a full abelian subcategory of the category of all $A$-modules which is stable under kernels, cokernels and extensions (weak Serre subcategory).
\end{enumerate}
\end{prop}

\begin{proof}
\cite[\href{https://stacks.math.columbia.edu/tag/091U}{Tag 091U}]{stacks-project}
\end{proof}

%Deri(If ${\mathrm L}^\bullet \to M^\bullet \to N^\bullet \to$ is a distinguished triangle and two of them are derived complete, then so is the third).

We assume from now on that $A$ has a finitely generated ideal of definition.

Let us first mention the important \emph{Derived Nakayama lemma}:

%%%%%%%%%%%%%%%%%%%
\begin{prop} \label{Nak}
If $M^\bullet$ is derived complete, then
\[
M^\bullet = 0 \Leftrightarrow M^\bullet \otimes_{A}^{\mathrm L} A/I = 0.
\]
\end{prop}

\begin{proof}
\cite[\href{https://stacks.math.columbia.edu/tag/0G1U}{Tag 0G1U}]{stacks-project}.
\end{proof}

%%%%%%%%%%%%%%%%
\begin{prop} \label{adder}
The inclusion of the category of derived complete complexes of $A$-modules into the category of all complexes of $A$-modules has a left adjoint $M^\bullet \mapsto \widehat {M^\bullet}$.
\end{prop}

\begin{proof}
\cite[\href{https://stacks.math.columbia.edu/tag/0920}{Tag 0920}]{stacks-project}
\end{proof}

The complex $\widehat{M^\bullet}$ is then called the \emph{derived completion} of $M^\bullet$.

\textbf{Warning}: If $M$ is an $A$-module, its derived completion may not be discrete, and even when it is discrete, it may differ from (classical) adic completion, as the following remarks show.
%Note however that we always have $\mathrm H^0(\mathrm H^0(M^\bullet)^\wedge) = \mathrm H^0(M^{\bullet, \wedge})$ from the zero part of the spectral sequence

%%%%%%%%%%%%%%
\begin{rmks}
Assume $(f_{1}, \ldots, f_{r})$ is an ideal of definition for $A$ and denote as above by $\mathrm{Kos}$ the corresponding Koszul complex (see \cite[\href{https://stacks.math.columbia.edu/tag/0623}{Tag 0623}]{stacks-project}).
Then,
\begin{enumerate}
\item If $M^\bullet$ is a complex of $A$-modules , we have
\[
\widehat {M^\bullet} = \mathrm R\varprojlim (M^\bullet \otimes^{\mathrm L}_{A} \mathrm{Kos}(A, f^n_{1}, \ldots, f^n_{r})).
\]
Moreover, $M^\bullet$ is derived complete if and only if $M^\bullet = \widehat M^\bullet$.
\item The \emph{derived completion} of an $A$-module $M$ will be
\[
\widehat {M} = \mathrm R\varprojlim \mathrm{Kos}(M, f^n_{1}, \ldots, f^n_{r}) \neq \mathrm \varprojlim M/(f^n_{1}, \ldots, f^n_{r}) 
\]
in general.
However, if $\widehat M$ is discrete and classically complete, then $\widehat M$ is identical to the classical adic completion of $M$.
\end{enumerate}
\end{rmks}

%%%%%%%%%%%%%%%
\begin{prop} \label{flatder}
\begin{enumerate}
\item If $M^\bullet$ is any complex of $A$-modules and $I$ denotes an ideal of definition for $A$, then
\[
M^\bullet \otimes_{A}^{\mathrm L} A/I \simeq \widehat{M^\bullet} \otimes_{A}^{\mathrm L} A/I.
\]
\item A complex of $A$-modules $M^\bullet$ is completely flat if and only if $\widehat {M^\bullet}$ is completely flat.
\end{enumerate}
\end{prop}

\begin{proof}
The second assertion is a consequence of the first one.
The first assertion results from the fact that both $M^\bullet \otimes_{A}^{\mathrm L} A/I$ and $\widehat{M^\bullet} \otimes_{A}^{\mathrm L} A/I$ are obviously derived complete and therefore share the same universal property.
\end{proof}

We will also write
\[
M^{\bullet,\wedge } := \widehat M^\bullet \quad \mathrm{and} \quad M^\bullet \widehat \otimes_{A}^{\mathrm L} N^\bullet := (M^\bullet \otimes_{A}^{\mathrm L} N^\bullet)^\wedge,
\]
as well as, if $M$ and $N$ are two $A$-modules,
\[
M \widehat \otimes_{A} N := \mathrm H^0(M \widehat \otimes^{\mathrm L}_{A} N).
\]
Note that $M \widehat \otimes_{A} N$ is \emph{not} the derived completion (which may not be discrete) of $M  \otimes_{A} N$ \emph{nor} its classical completion (because it may not be separated).

%%%%%%%%%%%%%%
\begin{lem} \label{tenschek}
If $M$ and $N$ are two $A$-modules, then
\[
M \widehat \otimes_{A} N \simeq \mathrm H^0((M \otimes_{A} N)^\wedge).
\]
\end{lem}

\begin{proof}
Let $M^\bullet$ be a complex of $A$-modules.
Recall from \cite[\href{https://stacks.math.columbia.edu/tag/0AAJ}{Tag 0AAJ}]{stacks-project} that if we assume that $\mathrm H^j(M^\bullet) = 0$ for $j > 0$, then the same holds for $\mathrm H^j(M^{\bullet,\wedge})$.
Now, there exists a bounded (in the sense of spectral sequences) spectral sequence (\cite[\href{https://stacks.math.columbia.edu/tag/0BKE}{Tag 0BKE}]{stacks-project})
\[
E_{2}^{i,j} = \mathrm H^i(\mathrm H^j(M^\bullet)^\wedge) \Rightarrow \mathrm H^{i+j}(M^{\bullet, \wedge}).
\]
It follows that we will always have $E_{2}^{i,j} = 0$ for $i > 0$ or $j > 0$.
As a consequence,
\[
\mathrm H^0(\mathrm H^0(M^\bullet)^\wedge) = \mathrm H^{0}(M^{\bullet, \wedge})
\]
and we can apply this to $M \otimes_{A}^{\mathrm L} N$.
\end{proof}

%%%%%%%%%%%%
\begin{rmks}
\begin{enumerate}
\item
If $M^\bullet$ is a complex of $A$-modules, we denote by $M^{\bullet,\mathrm{sep}}$ the (biggest) Hausdorff quotient of $\mathrm H^0(M^\bullet)$.
If $M$ is an $A$-module, then,
\[
\widehat M^\mathrm{sep} \simeq \varprojlim_{n \in \mathbb N} M/I^{n}
\]
is the \emph{classical completion} of $M$.
In particular, $ \widehat M \simeq \widehat M^\mathrm{sep}$ if and only if $\widehat M$ is discrete and classically complete.
\item Set
\[
M^\bullet \widehat \otimes^{\mathrm{sep}}_{A} N^\bullet := (M^\bullet \widehat \otimes^{\mathrm L}_{A} N^\bullet)^{\mathrm{sep}}.
\]
If $M$ and $N$ are two $A$-modules, then $M \widehat \otimes^{\mathrm{sep}}_{A} N$ is the classical completion of $M \otimes_A N$.
Moreover, 
$M \widehat \otimes^\mathrm{L}_{A} N$ is discrete and classically complete if and only if 
\[
M \widehat \otimes^\mathrm{L}_{A} N \simeq M \widehat \otimes^{\mathrm{sep}}_{A} N,
\]
in which case also $M \widehat \otimes^\mathrm{L}_{A} N \simeq M \widehat \otimes_{A} N$.
\end{enumerate}
\end{rmks}

%%%%%%%%%
\subsection{Example: the monogenic case}  \label{Append3}

We assume here that $A$ is endowed with the $f$-adic topology for some $f \in A$.

%%%%%%%%%
\begin{rmk} If $A$ is $f$-torsion free and $M$ is an $A$-module, then
\begin{enumerate}
\item $M \otimes^{\mathrm L}_{A} A/fA \simeq [M \overset f \to M]$,
\item $M \otimes^{\mathrm L}_{A} A/fA$ is discrete if and only if $M$ is $f$-torsion free,
\item $M $ is completely flat if and only if $M$ is $f$-torsion free and $M/fM$ is flat over $A/fA$.
\item If $M$ is formally flat and separated, then $M$ is $f$-torsion free.
\end{enumerate}
Also, if $M$ is an $f$-torsion free $A$-module, then $\widehat M$ is classically complete and $f$-torsion free.
\end{rmk}

It is however too much to require torsion freeness and we shall introduce a weaker condition.
We will denote the $f$-torsion of an $A$-module $M$ by\footnote{We want to avoid the notation $M[f]$ which may be confused with a twist or a shift.}
\[
{}_{f}M= \{s \in M, fs=0\}.
\]

%%%%%%%%%%%
\begin{dfn} \label{inftors}
\begin{enumerate}
\item The \emph{$f^\infty$-torsion} of an $A$-module $M$ is
\[
{}_{f^\infty}M:= \bigcup_{n\in \mathbb N}{}_{f^{n}}M.
\]
\item 
An $A$-module $M$ has $f^\infty$-torsion \emph{bounded} by $N \in \mathbb N$ if ${}_{f^\infty}M ={}_{f^{N}}M$.
It has \emph{bounded $f^\infty$-torsion} if if it has $f^\infty$-torsion bounded by $N$ for some $N \in \mathbb N$.
\item 
A complex of $A$-modules $M^\bullet$ has \emph{bounded $f^\infty$-torsion} if it is discrete and $\mathrm H^0(M^\bullet)$ has bounded $f^\infty$-torsion.
\end{enumerate}
\end{dfn}

For example, an $A$-module $M$ has bounded $f^\infty$-torsion when $M$ is noetherian, when $M$ is $f$-torsion free or when $M$ is $f^N$-torsion for some $N$.

Recall that an $A$-module $M$ is said to be completely flat if the corresponding complex of $A$-modules placed in degree zero satisfies the condition of definition \ref{compfl} with $I = (f)$.

%%%%%%%%%%
\begin{lem} \label{compbd}
If $M$ is an $A$-module, then the following conditions are equivalent:
\begin{enumerate}
\item $M$ is a completely flat $A$-module,
\item $M/fM$ is flat over $A/fA$ and $M \otimes^\mathrm{L}_{A} {}_{f}A \simeq {}_{f}M$.
\end{enumerate}
When this is the case, if $A$ has $f^\infty$-torsion bounded by $N$, then so does $M$.
\end{lem}

\begin{proof}
We know from remark \eqref{rmkplus} after definition \ref{compfl} that $M$ is completely flat if and only if $M/fM$ is flat over $A/fA$ and $M \otimes^\mathrm{L}_{A} A/fA \simeq M \otimes_{A} A/fA$.
Now, by definition, there always exists a distinguished triangle
\[
{}_{f}M[1] \to [M \overset {f} \to M] \to M/fM \to.
\]
From the case $M = A$, we deduce another distinguished triangle
\[
(M \otimes^\mathrm{L}_{A}{}_{f}A)[1] \to [M \overset {f} \to M] \to M \otimes_{A}^\mathrm{L} A/fA \to.
\]
It follows that
\[
M \otimes^\mathrm{L}_{A} {}_{f}A \simeq {}_{f}M \quad \Leftrightarrow \quad M \otimes_{A}^\mathrm{L} A/fA \simeq M/fM.
\]
As a consequence, we will automatically have $M \otimes_{A} {}_{f}A \simeq {}_{f}M$.
Applying this result to $f^n$ gives the last assertion.
\end{proof}

%%%%%%%%%%%%%
\begin{lem} \label{infbd}
\begin{enumerate}
\item If $M$ is an $A$-module with bounded $f^\infty$-torsion, then there exists an isomorphism of pro-complexes
\[
\left\{[M \overset {f^n} \to M]\right\}_{n \in \mathbb N} \simeq \{M/f^nM\}_{n \in \mathbb N}.
\]
\item If $A$ has bounded $f^\infty$-torsion and $M^\bullet$ is any complex of $A$-modules, then there exists an isomorphism of pro-complexes
\[
\left\{[M^\bullet \overset {f^n} \to M^\bullet]\right\}_{n \in \mathbb N} \simeq \left\{M^\bullet \otimes^\mathrm{L}_{A}A/f^nA\right\}_{n \in \mathbb N}.
\]
\end{enumerate}
\end{lem}

\begin{proof}
Recall that there exists a distinguished triangle
\[
{}_{f^n}M[1] \to [M \overset {f^n} \to M] \to M/f^nM \to.
\]
If the $f^\infty$-torsion of $M$ is bounded by $N$, we will have
\[
\left\{{}_{f^n}M\right\}_{n \in \mathbb N} \simeq \left\{{}_{f^N}M\right\}_{n \in \mathbb N} \simeq 0
\]
because the transition map is multiplication by $f$ (and therefore eventually $0$ on ${}_{f^N}M$).
We obtain the first isomorphism.
If we assume now that $A$ has bounded $f^\infty$-torsion and apply the first result to $A$, we obtain an isomorphism
\[
\left\{[M^\bullet \overset {f^n} \to M^\bullet]\right\}_{n \in \mathbb N} = \left\{M^\bullet \otimes^\mathrm{L}_{A}[A \overset {f^n} \to A]\right\}_{n \in \mathbb N} \simeq \{M^\bullet \otimes^\mathrm{L}_{A}A/f^nA\}_{n \in \mathbb N}. \qedhere
\] 
\end{proof}

%%%%%%%%%%%
\begin{rmk}
If $M^\bullet$ is a complex of $A$-modules, then its derived completion is
\[
\widehat {M^\bullet} = \mathrm R\varprojlim_{n} [M^\bullet \overset {f^n} \to M^\bullet].
\]
\end{rmk}

%%%%%%%%%%%%%
\begin{prop} \label{modeqn}
If $M$ has $f^\infty$-torsion bounded by $N$, then its derived completion $\widehat M$ is classically complete with $f^{\infty}$-torsion bounded by $N$.
\end{prop}

\begin{proof} It follows from the first assertion of lemma \ref{infbd} (and Mittag-Leffler argument) that derived and usual completions coincide.
Assume now that $f^ms_{n} \to 0$ when $n \to \infty$ for some $m \geq N$ and that the $f^\infty$-torsion of $M$ is bounded by $N$.
Then, given $k \in \mathbb N$, if we write $k' = k+m-N$, there exists $n_{0} \in \mathbb N$ such that for all $n \geq n_{0}$, we can write $f^ms_{n} = f^{k'}t_n$ with $t_n \in M$.
But then $f^m(s_{n} - f^{k'-m}t_n) = 0$ and therefore already $f^N(s_{n} - f^{k'-m}t_n) = 0$ so that $f^Ns_{n} = f^{k}t_n$.
This shows that $f^Ns_{n} \to 0$.
Thus we see that $\widehat M$ also has $f^\infty$-torsion bounded by $N$.
%\footnote{How about a more conceptual proof ?}.
\end{proof}

%%%%%%%%%%%%%
\begin{prop} \label{modeq2}
If $A$ has bounded $f^\infty$-torsion and $M$ is an $A$-module, then the following are equivalent
\begin{enumerate}
\item \label{cond1} $M$ is completely flat,
\item \label{cond2} $M$ is formally flat with bounded $f^\infty$-torsion.
\end{enumerate}
\end{prop}

\begin{proof} Since complete flatness implies formal flatness, we already know from lemma \ref{compbd} that condition \ref{cond1} implies condition \ref{cond2}.
Conversely, using both assertions in lemma \ref{infbd}, we will have an isomorphism of pro-complexes
\[
\{(M \otimes^\mathrm{L}_{A}A/f^nA) \otimes^\mathrm{L}_{A/f^nA} A/fA\}_{n \in \mathbb N} \simeq \{M/f^nM \otimes^\mathrm{L}_{A/f^nA} A/fA\}_{n \in \mathbb N}.
\]
But
\[
(M \otimes^\mathrm{L}_{A} A/f^nA) \otimes^\mathrm{L}_{A/f^nA} A/fA \simeq M \otimes^\mathrm{L}_{A} A/fA
\]
and, since $M/f^nM$ is flat over $A/f^nA$,
\[
M/f^nM \otimes^\mathrm{L}_{A/f^nA} A/fA \simeq M/f^nM \otimes_{A/f^nA} A/fA = M/fM.
\]
It follows that $M \otimes^\mathrm{L}_{A} A/fA \simeq M/fM$ is discrete flat over $A/fA$.
\end{proof}

%%%%%%%%%%
\begin{prop} \label{monobd}
If $A$ has $f^{\infty}$-torsion bounded by $N$ and $M$ is a formally flat $A$-module which is separated, then $M$ also has $f^{\infty}$-torsion bounded by $N$.
\end{prop}

\begin{proof}
Since $M$ is separated, it is contained in its classical completion and we may therefore assume from the begining that $M$ is actually classically complete.
There exists, for $k,n \geq N$, a commutative diagram with exact row and columns
\[
\xymatrix{&& & 0 \ar[d] & 0 \ar[d] 
\\ && & {}_{f^N}A \ar[d] \ar@{=}[r] &  {}_{f^N}A \ar[d] 
\\ 0 \ar[r]  & {}_{f^N}A \ar[r] \ar@{=}[d] & A \ar[r]^-{f^n} \ar@{=}[d]  & A \ar[r] \ar[d]^{f^k} & A/f^{n}\ar[d]^{f^k}  \ar[r] & 0
\\ 0 \ar[r] & {}_{f^N}A \ar[r] &A \ar[r]^-{f^{n+k}} & A \ar[r] \ar[d] & A/f^{n+k} \ar[r] \ar[d] & 0
\\ &&  & A/f^k \ar@{=}[r] \ar[d] & A/f^k \ar[d] 
\\ &&& 0 & 0.}
\]
In particular, there exists an exact sequence of $A/f^{n+k}$-modules
\[
0 \to  {}_{f^N}A \to A/f^n \overset {f^k} \to A/f^{n+k} \to A/f^k \to 0.
\]
Since $M$ is formally flat, $M/f^{n+k}M$ is flat over $A/f^{n+k}$, and therefore, the sequence
\[
0 \to M \otimes_A {}_{f^N}A  \to M/f^nM \overset {f^k} \to M/f^{n+k}M \to M/f^kM \to 0
\]
is also exact.
Taking inverse limit provides us with a left exact sequence
\[
0 \to M \otimes_A  {}_{f^N}A \to M \overset {f^k} \to M,
\]
showing that ${}_{f^k} M = M \otimes_A  {}_{f^N}A$ does not depend on $k$.
\end{proof}

%%%%%%%%%%%
\begin{prop} \label{monoth}
Assume $A$ has bounded $f^\infty$-torsion.
If $M^\bullet$ be a complex of $A$-modules, then, the following conditions are equivalent
\begin{enumerate}
\item $M^\bullet$ is derived complete and completely flat,
\item $M^\bullet$ is classically complete and formally flat.
\end{enumerate}
\end{prop}

\begin{proof}
Assume that $M^\bullet$ is derived complete and completely flat.
Since $A$ has bounded $f^\infty$-torsion, it follows from lemma \ref{infbd} and the fact that $M^\bullet$ is derived complete, that there exists an isomorphism
\[
M^\bullet = \mathrm R\varprojlim_{n} [M^\bullet \overset {f^n}\to M^\bullet] \simeq \mathrm R \varprojlim_{n} M^\bullet \otimes^\mathrm{L}_{A} A/f^nA.
\]
Since $M^\bullet$ is completely flat and $f^nA/f^{n+1}A$ is an $A/f^nA$-module, the complex
\[
M^\bullet \otimes_{A}^\mathrm{L} (f^nA/f^{n+1}A)
\]
is discrete, and this implies the surjectivity of the map
\[
\mathrm H^0(M^\bullet \otimes_{A}^\mathrm{L} A/f^{n+1}A) \twoheadrightarrow \mathrm H^0(M^\bullet \otimes_{A}^\mathrm{L} A/f^{n}A).
\]
Using again the fact that $M^\bullet$ is completely flat, we also know that
\[
M^\bullet \otimes^\mathrm{L}_{A} A/f^nA \simeq \mathrm H^0(M^\bullet \otimes^\mathrm{L}_{A} A/f^nA).
\]
It follows that
\[
M^\bullet \simeq \varprojlim \mathrm H^0(M^\bullet \otimes^\mathrm{L}_{A} A/f^nA)
\]
is discrete.
The complex $M^\bullet$ being discrete and completely flat, we have
\[
\mathrm H^0(M^\bullet \otimes_{A}^\mathrm{L} A/f^nA) = \mathrm H^0(M^\bullet)/f^n\mathrm H^0(M^\bullet).
\]
It follows that
\[
\mathrm H^0(M^\bullet) \simeq M^\bullet \simeq \varprojlim \mathrm H^0(M^\bullet)/f^n\mathrm H^0(M^\bullet)
\]
is classically complete.
Everything else follows from the previous results.
\end{proof}

%%%%%%%%%%%%%%%%%%%%
\subsection{Example: the ``bigenic'' case}  \label{Append4}

We assume here that $A$ is endowed with the $(f, g)$-adic topology for some fixed $f, g\in A$.

%%%%%%%%%%%%%%%%
\begin{dfn} \label{defbound}
\begin{enumerate}
\item
An $A$-module $M$ is said to be \emph{bounded} (with respect to $f$ modulo $g$) if $M$ is $g$-torsion free and $M/gM$ has bounded $f^\infty$-torsion.
\item 
A complex of $A$-modules $M^\bullet$ is said to be \emph{bounded} if it is discrete and $\mathrm H^0(M^\bullet)$ is bounded.
\end{enumerate}
\end{dfn}

We will also say that $A$ is \emph{bounded} when it is bounded as an $A$-module.

%%%%%%%%%%%%%%%
\begin{rmks}
\begin{enumerate}
\item
This notion obviously depends on the ordered pair $(f,g)$, but the properties will also hold if we replace any of them by some power.
\item
If $M$ is a $g$-torsion free $A$-module, then we have
\[
\mathrm{Kos}(M, f, g)
:=
\left[
\begin{gathered}
\xymatrix{M \ar[r]^{g} & M \\ M \ar[r]^{g} \ar[u]^{f} & M \ar[u]^{f}}
\end{gathered}
\right] \simeq [M/gM \overset {f} \to M/gM].
\]
\end{enumerate}
\end{rmks}

%%%%%%%%%%%%%%%
\begin{lem} \label{bdlem}
\begin{enumerate}
\item
If $M$ is a bounded $A$-module, then there exists an isomorphism of pro-complexes
\[
\left\{\mathrm{Kos}(M, f^n, g^m)\right\}_{n,m \in \mathbb N} \simeq \{M/(f^n, g^m)\}_{n,m \in \mathbb N}.
\]
\item
If $A$ is bounded and $M$ is any $A$-module, then there exists an isomorphism of pro-complexes
\[
\left\{\mathrm{Kos}(M, f^n, g^m)\right\}_{n,m \in \mathbb N} \simeq \{M \otimes^\mathrm{L}_{A} A/(f^n, g^m)\}_{n,m \in \mathbb N}.
\]
\end{enumerate}
\end{lem}

\begin{proof}
Using the last remark, this follows from lemma \ref{infbd}.
\end{proof}

%%%%%%%%%%%%%%%%%
\begin{lem} \label{compbound}
If $M$ is an $A$-module such that $M/gM$ has bounded $f^\infty$-torsion, then the $(f,g)$-adic topology on $gM$ is identical to the topology induced by the $(f,g)$-adic topology of $M$.
\end{lem}

\begin{proof}
It is actually sufficient to show that the $f$-adic topology on $gM$ is identical to the topology induced by the $f$-adic topology of $M$.
In other words, we have to show that, given $n \in \mathbb N$, we can find $m \in \mathbb N$ such that
\[
gM \cap f^mM \subset f^ngM.
\]
If $M/gM$ has $f^\infty$-torsion bounded by $N \in \mathbb N$, then we can choose $m := n+N$: if $s \in M$ satisfies $f^ms = gt$ for some $t \in M$, then there also exists $r \in M$ such that $f^Ns = gr$, and therefore $f^ms=f^ngr$.
\end{proof}

%%%%%%%%%%%%%
\begin{prop} \label{compcomp}
If an $A$-module $M$ is bounded, then its derived completion $\widehat M$ is classically complete and bounded.
\end{prop}

\begin{proof}
It follows from the second assertion of lemma \ref{bdlem} that both completions coincide and it remains to show that $\widehat M$ is bounded.
It follows from lemma \ref{compbound} that the sequence
\[
0 \longrightarrow M \overset {g} \longrightarrow M \longrightarrow M/gM \longrightarrow 0
\]
is \emph{strict} exact, and then the sequence
\[
0 \longrightarrow \widehat M \overset {g} \longrightarrow \widehat M \longrightarrow \widehat {M/gM} \longrightarrow 0
\]
 is also strict exact.
As a consequence, we see that $\widehat M$ is $g$-torsion free and $\widehat M/g\widehat M = \widehat {M/gM} $ has bounded $f^\infty$-torsion thanks to proposition \ref{modeqn}.
 \end{proof}

%%%%%%%%%%%%%
\begin{prop} \label{flatflat}
If $A$ is bounded and $M$ is an $A$-module, then the following are equivalent
\begin{enumerate}
\item $M$ is completely flat,
\item $M$ is formally flat and bounded.
\end{enumerate}
\end{prop}

\begin{proof}
This is proven exactly as in proposition \ref{modeq2}.
\end{proof}

%%%%%%%%%%%
\begin{prop} \label{torsfr}
If $A$ is bounded and $M$ is a formally flat $A$-module which is separated, then $M$ also is bounded.
\end{prop}

\begin{proof}
If $A/g$ has $f^\infty$-torsion bounded by $N$, then it follows from the snake-lemma
applied to the diagram
\[
\xymatrix{
0 \ar[r]  & A/g^n \ar[r]^-{g} \ar[d]^{f^k} & A/g^{n+1} \ar[r] \ar[d]^{f^k} & A/g\ar[d]^{f^k}  \ar[r] & 0
\\ 0 \ar[r] & A /g^n \ar[r]^-{g} & A/g^{n+1} \ar[r] & A/g \ar[r] & 0
}
\]
that there exists for $k \geq N$ and any $n$, a long exact sequence of $A/(f^{k},g^{n+1})$-modules
\[
\cdots \to  {}_{f^N}(A/g) \overset {\alpha_k} \to A/(f^{k},g^n) \overset {g} \to A/(f^{k},g^{n+1})\to \cdots
\]
and a quick look shows that $\alpha_{k} = f^{k-N}\alpha_N$.
Since $M$ is formally flat, $M/(f^{k},g^{n+1})M$ is flat over $A/(f^{k},g^{n+1})$, and therefore, the sequence
\[
\cdots \to M \otimes_A {}_{f^N}(A/g) \to M/(f^{k},g^n)M \overset {g} \to M/(f^{k},g^{n+1})M \to \cdots
\]
is also exact.
Moreover, the image of the first map is contained in $f^{k-N}M/(f^{k},g^n)M$.
It follows that, whenever $s \in M$ satisfies $gs\in (f^{k},g^{n+1})M$, we have $s \in (f^{k-N},g^n)M$.
In particular, if $gs=0$, then, necessarily, $s =0$ because $M$ is separated.
This shows that $M$ is $g$-torsion free.
But it also implies that
\[
gM \cap (f^{k},g^{n+1})M \subset (f^{k-N},g^n)gM
\]
and $gM$ is therefore closed in $M$.
In other words, $M/gM$ is separated and we may apply proposition \ref{monobd} to conclude that $M$ is bounded.
\end{proof}

%%%%%%%%%%%%%%%%%
\begin{thm} \label{compform}
Assume $A$ is bounded and let $M^\bullet$ be a complex of $A$-modules.
Then, the following conditions are equivalent
\begin{enumerate}
\item $M^\bullet$ is derived complete and completely flat,
\item $M^\bullet$ is classically complete and formally flat.
\end{enumerate}
\end{thm}

\begin{proof}
We already know that the second condition implies the first.
Assume conversely that $M^\bullet$ is derived complete and completely flat.
Then, $M^\bullet$ is automatically formally flat.
Moreover, the fact that $M^\bullet$ is completely flat implies that $M^\bullet \otimes^\mathrm{L}_{A} A/g^nA$ is completely flat (for the $f$-adic topology).
On the other hand, since $A$ is $g$-torsion free, we have
\[
M^\bullet \otimes^\mathrm{L}_{A} A/g^nA \simeq [M^\bullet \overset {g^n} \to M^\bullet].
\]
Since $M^\bullet$ is derived complete, proposition \ref{stb} implies that $M^\bullet \otimes^\mathrm{L}_{A} A/g^nA$ also is derived complete (for the $f$-adic topology).
Therefore, it follows from proposition \ref{monoth} that $M^\bullet \otimes^\mathrm{L}_{A} A/g^nA$ is (discrete) classically complete and formally flat with bounded $f^\infty$-torsion.
In particular,
\[
M^\bullet \otimes^\mathrm{L}_{A} A/g^nA \simeq \mathrm H^0(M^\bullet \otimes^\mathrm{L}_{A} A/g^nA)
\]
is discrete flat over $A/g^nA$.
Since $\mathrm H^0(M^\bullet \otimes^\mathrm{L}_{A} A/g^{n+1}A)$ is flat over $A/g^{n+1}A$, we have the following commutative diagram
\[
\xymatrix{
\mathrm H^0(M^\bullet \otimes^\mathrm{L}_{A} A/g^{n+1}A) \otimes^\mathrm{L}_{A} A/g^nA \ar[r]^-\simeq \ar[d]^-\simeq & \mathrm H^0(M^\bullet \otimes^\mathrm{L}_{A} A/g^{n+1}A)/g^n\mathrm H^0(M^\bullet \otimes^\mathrm{L}_{A} A/g^{n+1}A) \ar[dd] \\
(M^\bullet \otimes^\mathrm{L}_{A} A/g^{n+1}A) \otimes^\mathrm{L}_{A} A/g^nA \ar[d]^-\simeq \ar[d]^-\simeq\\
M^\bullet \otimes^\mathrm{L}_{A} A/g^nA \ar[r]^-\simeq & \mathrm H^0(M^\bullet \otimes^\mathrm{L}_{A} A/g^{n}A),
}
\]
from which we deduce the surjectivity of the map
\[
\mathrm H^0(M^\bullet \otimes^\mathrm{L}_{A} A/g^{n+1}A) \twoheadrightarrow \mathrm H^0(M^\bullet \otimes^\mathrm{L}_{A} A/g^{n}A).
\]
In particular, this is a Mittag-Leffler system.
Note now that
\[
M^\bullet \otimes^\mathrm{L}_{A} [A \overset {g^n} \to A] \simeq [M^\bullet \overset {g^n} \to M^\bullet] \simeq M^\bullet \otimes^\mathrm{L}_{A} A/g^nA \simeq \mathrm H^0(M^\bullet \otimes^\mathrm{L}_{A} A/g^nA).
\]

Since $M^\bullet$ is derived complete (for the $g$-adic topology), we obtain that
\[
M^\bullet = \mathrm R\varprojlim M^\bullet \otimes^\mathrm{L}_{A} [A \overset {g^n} \to A] \simeq \mathrm \varprojlim \mathrm H^0(M^\bullet \otimes^\mathrm{L}_{A} A/g^nA)
\]
is discrete and classically complete.
In particular,
\[
[\mathrm H^0(M^\bullet) \overset g \to \mathrm H^0(M^\bullet)] \simeq M^\bullet \otimes^\mathrm{L}_{A} A/gA
\]
is discrete.
In other words, $\mathrm H^0(M^\bullet)$ is $g$-torsion free and
\[
\mathrm H^0(M^\bullet)/g \mathrm H^0(M^\bullet) \simeq M^\bullet \otimes^\mathrm{L}_{A} A/gA
\]
has bounded $f^\infty$-torsion.
\end{proof}

%%%%%%%%%%%%%%%
\begin{rmks}
\begin{enumerate}
\item
Theorem \ref{compform} is the same as Tian's proposition 1.4 in \cite{Tian21}.
We are very thankful to him for clarifying some points in his proof allowing us to improve on our original statement.
\item As a consequence of theorem \ref{compform}, we see that if $B$ is a bounded $A$-algebra and $M$ is a completely flat $A$-module, then the derived completed tensor product $B \widehat \otimes^{\mathrm L}_{A} M$ is identical to the classical completed tensor product $B \widehat \otimes_{A} M$.
\end{enumerate}
\end{rmks}

\addcontentsline{toc}{section}{References}
\printbibliography

\Addresses

\end{document}